\documentclass{amsart}
\usepackage[T1]{fontenc}
\usepackage{geometry}
\usepackage{color}
\usepackage{amsmath}
\usepackage{amssymb}
\usepackage{graphicx}
\usepackage{esint}
\usepackage{amsmath}
\usepackage{graphicx}%
\usepackage{amsfonts}%
\usepackage{amssymb}
\usepackage{epic,eepic}
\usepackage[T1]{fontenc}
\usepackage[latin9]{inputenc}
\usepackage{mathrsfs}
\usepackage{amsthm}
\usepackage{amstext}
\usepackage{amssymb}
\usepackage{graphicx}
\usepackage{esint}

\makeatletter

\providecommand{\tabularnewline}{\\}

\numberwithin{equation}{section}
\numberwithin{figure}{section}

\usepackage{amsthm}
\usepackage{amscd}
\usepackage{amsfonts}
\usepackage{amsbsy}
\usepackage{latexsym}
\usepackage{mathrsfs}
\usepackage{color}\usepackage{enumerate}
\usepackage{color}\usepackage{float}

\theoremstyle{plain}
\newtheorem{theorem}{Theorem}[section]
\newtheorem{lemma}{Lemma}[section]

\theoremstyle{definition}
\newtheorem{definition}{Definition}[section]

\theoremstyle{remark}
\newtheorem{remark}{Remark}[section]

\newcommand{\diam}{\operatorname{diam}}

\newcommand{\dD}{\operatorname{D}}
\newcommand{\dG}{\operatorname{G}}
\newcommand{\length}{\operatorname{length}}
\newcommand{\area}{\operatorname{area}}
\newcommand{\Span}{\operatorname{Span}}

\newcommand{\bD}{\mathbb{D}}
\newcommand{\bE}{\mathbb{E}}
\newcommand{\bI}{\mathbb{I}}

\newcommand{\bR}{\mathbb{R}}

\providecommand{\ip}[1]{\langle#1\rangle}
\providecommand{\Ip}[1]{\left\langle#1\right\rangle}
\providecommand{\bigip}[1]{\bigl\langle#1\bigr\rangle}
\providecommand{\set}[1]{\{#1\}}
\providecommand{\Set}[1]{\left\{#1\right\}}

\providecommand{\abs}[1]{\lvert#1\rvert}
\providecommand{\Abs}[1]{\left\lvert#1\right\rvert}

\providecommand{\norm}[1]{\lVert#1\rVert}

\renewcommand{\vec}[1]{\boldsymbol{#1}}

\makeatother
\begin{document}
	\title[Stable and convergent method for Hodge decomposition]{A stable and convergent method for Hodge decomposition of fluid-solid  interaction}
\author{Gangjoon Yoon}
\address{National Institute for Mathematical Sciences, Daejeon 34047, Korea}
\email{yoon@nims.re.kr}
\author{Chohong Min$^{\dag}$}
\address{Department of Mathematics, Ewha Woman's University, Seoul 03760,
	Korea}
\email{chohong@ewha.ac.kr}
\thanks{$^{\dag}$ Corresponding author. C. Min was supported by Basic Science
	Research Program through the National Research Foundation of Korea(NRF)
	funded by the Ministry of Education(2009-0093827).}
\author{Seick Kim}
\address{Department of Mathematics, Yonsei University, Seoul 03722, Korea}
\email{kimseick@yonsei.ac.kr}
\thanks{S. Kim is supported by NRF-20151009350.}
	
	\date{}
\subjclass[2000]{Primary 76D03, 65N06, 76M20; Secondary 35Q30, 35J25}
\keywords{Fluid-solid interaction, Helmholtz-Hodge decomposition, Augmented
	Hodge decomposition, Numerical analysis}
	
\begin{abstract}
Fluid-solid interaction has been a challenging subject due to their 
strong nonlinearity and multidisciplinary nature. Many of the numerical methods for solving FSI problems have struggled with non-convergence and numerical instability.  In spite of comprehensive studies, it has been still a challenge to develop a method that guarantees both convergence and stability. 

Our discussion in this work is restricted to the interaction of viscous incompressible fluid flow and a rigid body. We take the monolithic approach by Gibou and Min \cite{GibouMin} that results in an extended Hodge projection. The projection updates not only the fluid vector field but also the solid velocities.
We derive the equivalence of the extended Hodge projection to the Poisson equation with non-local Robin boundary condition. 
We prove the existence, uniqueness, and regularity for the weak solution of the Poisson equation, through which the Hodge projection is shown to be unique and orthogonal. Also, we show
the stability of the projection in a sense that the projection does not increase the total kinetic energy of fluid and solid. Also, we discusse a numerical
method  as a discrete analogue to the Hodge projection, then we
show that the unique decomposition and orthogonality also hold in
the discrete setting. As one of our main results, we prove that the numerical solution is convergent with at least the first order accuracy. We carry out numerical experiments in two and three dimensions, which validate our analysis and arguments. 

\end{abstract}
	\maketitle

\section{Introduction}
In this article, we consider the interaction of fluid and structure. An immersed structure in fluid interacts with fluid in two ways. On their interface, 
fluid cannot penetrate into structure and the motion of structure is affected by the normal
stress of fluid. Fluid-structure interaction (FSI) has been a challenging subject due to their
strong nonlinearity and multidisciplinary nature (\cite{Chakrabarti}, \cite{Dowell:Hall}, \cite{Morand:Ohayon}). In many scientific and engineering
areas, FSI problems play prominent roles, but it is hard or impossible to find exact solutions to the problems, so that
their solutions should have been approximated by numerical solutions.

For approximating FSI problems, numerous numerical methods have been proposed. Most of them
have struggled with non-convergence (\cite{Borazjani:Ge:Sotiropoulos}, \cite{Causin:Gerbeau:Nobile}, \cite{Forster:Wall:Ramm}, \cite{Guermond:Minev:Shen}) and numerical instability \cite{Uhlman}. Only few of them have guaranteed stability (\cite{Bridson}, \cite{GibouMin}, \cite{Gretarsson:Kwatra:Fedkiw}). 
In spite of comprehensive studies, it has been still
a challenge to develop a method that guarantees both convergence and stability.

There are two main approaches to obtain a numerical method for the FSI problems : Monolithic and partitioned approaches. The monolithic approach treats the fluid and structure dynamics in the same mathematical framework, and the partitioned approach treats them separately. The  monolithic approach \cite{Hubner:Walhorn:Dinkler, Michler:et:al:2004, Ryzhakov:et:al:2010, Walhorn:Kolke:Hubner:Dinkler} solves simultaneously 
 the linearized equations of fluid-structure coupling conditions in one system. Even though this approach achieves better accuracy for the multidisciplinary problem, usually it requires well-designed preconditioners and costs quite expensive computational time \cite{Badia:et:al:2008, Gee:et:at:2011, Heil:Hazel:Boyle}. To reduce the computational time, the partitioned approach treats the structure and the fluid as the two physical fields and solves them separately. Even though the partitioned methods have been extensively studied and developed \cite{Badia:Nobile:Vergara:2008, Bukac:Yorov:Zunino, Degroote:et:al:2008, Farhat:Zee:Geuzaine, Nobile:Vergara}, this approach still requires the construction of efficient schemes to produce stable, accurate results. Despite the recent developments, only a few partitioned methods have guaranteed the convergence \cite{Fernandez, Fernandez:Mullaert, Muha:Canic}.   

Our discussion is restricted to the interaction of viscous incompressible fluid flow and a rigid body.
The incompressible fluid flow is governed by the Navier-Stokes equations and the motion rigid body is governed by the Euler equations. On their interface $\Gamma(t),$ fluid and solid interacts in two ways. They should have the same normal velocity component and fluid delivers the net force of normal stress to solid. Combining the governing equations and taking into consideration the two-way couplings, we have the following two syatems of equations for fluid and solid:
consisting of the system of momentum equation and incompressible condition:
\begin{equation}
\left\{ \begin{aligned}\rho(\vec{U}_{t}+\vec{U}\cdot\nabla\vec{U}) & =-\nabla p+\nabla\cdot\left(2\mu\bD\right)\quad\text{in }\;\Omega\\
\nabla\cdot\vec{U} & =0\quad\text{in }\;\Omega.
\end{aligned}
\right.\label{eq:Navier_Stokes}
\end{equation}
and
\begin{equation}
\left\{ 
\begin{aligned}
\dot{\vec{c}} & =\vec{v}\\
m\dot{\vec{v}} & =f\\
\bI\dot{\vec{\omega}} & =\tau. 
\end{aligned}
\right.\label{eq:Solid_eqs}
\end{equation}
Here, $\vec{U}$ is the flow vector field and 
$\vec{v}(t)$ and $\vec{\omega}(t)$
denote the linear and angular velocities at the center of mass $\vec{c}(t)$
of the rigid body, respectively. With $\vec{J}:=(\vec{x}-\vec{c})\times\vec{n},$ note that $
\vec{U}\cdot\vec{n}=\left(\vec{v}+\vec{\omega}\times(\vec{x}-\vec{c})\right)\cdot\vec{n}$ describes the non-penetration boundary condition at the interface. The other notations follow the standard settings and we refer to \cite{GibouMin} for more details.

The incompressible condition prevents the development of a discontinuity,
and a standard approach to approximate the time derivatives in the
system of equations \eqref{eq:Navier_Stokes}and \eqref{eq:Solid_eqs}, and  is the second order SL-BDF method \cite{GibouMin, Xiu:Karniadakis},
which can be written in the followoing two steps. The first step is
to solve the system with a pressure guess.
\[
\left\{ \begin{array}{rcl}
\rho\frac{\frac{3}{2}U^{*}-2U^{n}+\frac{1}{2}U^{n-1}}{\Delta t} & = & -\nabla p^{n}+\nabla\cdot\left(2\mu\mathbb{D}^{*}\right)\textrm{ in }\Omega^{n+1}\\
\frac{\frac{3}{2}c^{*}-2c^{n}+\frac{1}{2}c^{n-1}}{\Delta t} & = & 2v^{n}-v^{n-1}\\
m\frac{\frac{3}{2}v^{*}-2v^{n}+\frac{1}{2}v^{n-1}}{\Delta t} & = & \intop_{\Gamma^{n+1}}\left(p^{n}-2\mu\mathbb{D}^{*}\right)n\, ds\\
\mathbb{I}\frac{\frac{3}{2}\omega^{*}-2\omega^{n}+\frac{1}{2}\omega^{n-1}}{\Delta t} & = & \intop_{\Gamma^{n+1}}\left(x-c^{n+1}\right)\times\left(p^{n}-2\mu\mathbb{D}^{*}\right)n\, ds
\end{array}\right.
\]
The second step is basically to enforce the divergence-free condition
and impose the two-way coupled interactions.
\[
\left\{ \begin{array}{rclc}
\rho U^{*} & = & \rho U^{n+1}+\nabla q^{n+1} & \textrm{ in }\Omega^{n+1}\\
\nabla\cdot U^{n+1} & = & 0 & \textrm{ in }\Omega^{n+1}\\
U^{n+1}\cdot n & = & v^{n+1}\cdot n+\omega^{n+1}\cdot J & \textrm{ in }\Gamma^{n+1}\\
mv^{*} & = & mv^{n+1}-\intop_{\Gamma^{n+1}}q^{n+1}n\, ds\\
\mathbb{I}\omega^{*} & = & \mathbb{I}\omega^{n+1}-\intop_{\Gamma^{n+1}}q^{n+1}J\, ds
\end{array}\right.
\]
The conventional Hodge decomposition splits a vector field into a unique sum of the divergence-free vector field and the gradient field. We take the monolithic approach that updates not only the fluid vector field but also the solid velocities. As we shall discuss in details in Section \ref{sec2}, the monolithic extension of the conventional Hodge projection that are denoted by
\[
\left(\vec{U}^{n+1},\vec{v}^{n+1}, \vec{\omega}^{n+1}\right)=
\mathcal{P}\left( \vec{U}^{*},\vec{v}^{*},\vec{\omega}^{*}\right).
\]
We derive the equivalence of the extended Hodge projection to the Poisson equation with non-local Robin boundary condition. In Section \ref{sec3}, we prove the existence, uniqueness, and regularity for the weak solution of the Poisson equation, through which the Hodge projection is shown to be unique and orthogonal. Also, we show
the stability of the projection in a sense that the projection does not increase the total kinetic energy of fluid and solid. Section~\ref{sec4} discusses the numerical
method \cite{GibouMin} as a discrete analogue to the Hodge projection, and
shows that the unique decomposition and orthogonality also hold in
the discrete setting. In Section \ref{sec5}, we prove that the numerical solution is convergent with at least the first order accuracy. In Section \ref{sec6}, we carry out numerical experiments in two and three dimensions. The numerical tests validate our analysis and arguments. 

\section{Monolithic decomposition}\label{sec2}

Let $\Omega\subset\bR^{d}$ ($d=2,3$) be a bounded domain (i.e.,
open connected set) with boundary $\partial\Omega=\Gamma\cup\Gamma'$,
where $\Gamma\cap\Gamma'=\emptyset$. We assume that the domain $\Omega$
is occupied by an incompressible fluid and $\Gamma$ is the interface
with the fluid and a solid; the other part $\Gamma'$ of the boundary 
is always fixed. Assume that a fluid vector field $\vec{U}^{*}:\Omega\to\bR^{d}$
and a solid linear velocity $\vec{v}^{*}\in\bR^{d}$ and angular velocity
$\vec{\omega}^{*}\in\bR^{d}$ are given. After the fluid-solid interaction
and the incompressible condition applied, the fluid vector field becomes
solenoidal and the fluid vector field and solid vector field have
the same normal component on the fluid-solid interface $\Gamma$.

\begin{figure}
	\begin{centering}
		\includegraphics[scale=0.5]{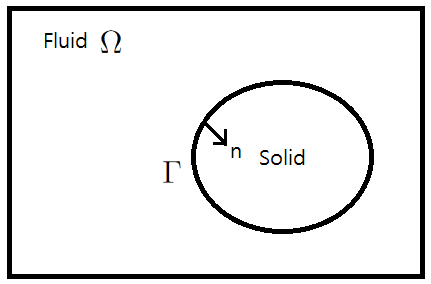} 
		\par\end{centering}
	
	\protect\caption{Fluid-solid interaction}
\end{figure}

In this work, we try to decompose $(\vec{U}^{*},\vec{v}^{*},\vec{\omega}^{*})$
as 
\begin{equation}
\left\{ \begin{aligned}\rho\vec{U}^{*} & =\rho\vec{U}+\nabla p\quad\text{in }\;\Omega\\
m\vec{v}^{*} & =m\vec{v}-{\displaystyle \int_{\Gamma}p\vec{n}\, dS}\\
\bI\vec{\omega}^{*} & =\bI\vec{\omega}-{\displaystyle \int_{\Gamma}p\vec{J}\, dS}
\end{aligned}
\right.\label{eq-main-decomposition}
\end{equation}

\noindent with a vector field $\vec{U}:\Omega\to\mathbb{R}^{d}$ and
a scalar field $p:\Omega\to\mathbb{R}$ and velocities $\vec{v},\vec{\omega}\in\mathbb{R}^{d}$
that satisfy the incompressible condition and the non-penetration
condition

\begin{equation}
\left\{ \begin{aligned}\nabla\cdot\vec{U} & =0\quad\text{in }\;\Omega\\
\vec{U}\cdot\vec{n} & =\vec{v}\cdot\vec{n}+\vec{\omega}\cdot\vec{J}\quad\text{on }\;\Gamma\\
\vec{U}\cdot\vec{n} & =0\quad\text{on }\;\Gamma'.
\end{aligned}
\right.\label{eq-boundary-Cdns}
\end{equation}

When we find a scalar function $p$ in \eqref{eq-main-decomposition}, the decomposition is achieved monolithically by setting $(\vec{U},\vec{v},\vec{\omega})$ as
$$
(\vec{U},\vec{v},\vec{\omega}):=(\vec{U}^{*},\vec{v}^{*},\vec{\omega}^{*})-\left(\frac{1}{\rho}\nabla p,\;-\frac{1}{m}\int_{\Gamma}p\vec{n}\, dS,\;-\mathbb{I}^{-1}\int_{\Gamma}p\vec{J}\, dS\right).
$$
We call $(\vec{U},\vec{v},\vec{\omega})$ the augmented Hodge projection of the state variable $(\vec{U}^{*},\vec{v}^{*},\vec{\omega}^{*}),$
$$
(\vec{U},\vec{v},\vec{\omega})=\mathcal{P} \left( \vec{U}^{*},\vec{v}^{*},\vec{\omega}^{*} \right),
$$
and we will check the stability of the projection by estimating the kinetic energies and 
show the orthogonality of the decomposition \eqref{eq-main-decomposition} with respect to the inner product induced by the kinetic energy.

The existence and uniqueness of the decomposition will be shown in the following section where we will find $p$ by solving a Poisson equation and study more details on the equation. 

Throughout this work, $\rho$ is the fluid density, which is assumed
to be constant; $m$ is the mass of a rigid body whose boundary is
$\Gamma$; $\vec{c}$ is the center of mass of the rigid body; $\bI\in\bR^{d\times d}$
is a symmetric positive definite matrix, the inertia tensor of the
rigid body. And $\vec{n}$ is the unit normal vector field of $\Gamma$
and $\vec{J}$ is a vector field defined on $\Gamma$ by $\vec{J}(\vec{x})=(\vec{x}-\vec{c})\times\vec{n}$
for $\vec{x}\in\Gamma$.

In case $d=2,$ we regard vectors in $\mathbb{R}^{2}$ as vectors
in $\mathbb{R}^{3}$ by a trivial extension as follows. The angular
velocity $\vec{\omega}$ reads as $\vec{\omega}=(0,0,\omega)$, and
for $\vec{a}=(a_{1},a_{2})$ and $\vec{b}=(b_{1},b_{2}),$ the cross
product of $\vec{a}$ and $\vec{b}$ is defined to be a scalar $a_{1}b_{2}-a_{2}b_{1}$
or the vector on $z$-axis given by $\vec{a}\times\vec{b}=(a_{1},a_{2},0)\times(b_{1},b_{2},0)$.
In this case, however, the vector calculation is fulfilled based on
$\mathbb{R}^{2}$, which may cause no confusion.

We begin by introducing the following lemma, which plays an essential role in this work.

\begin{lemma} \label{lemma-zero-nj} In the above setting, $\int_{\Gamma}\vec{n}\, dS=0$
	and $\int_{\Gamma}\vec{J}\, dS=0$. \end{lemma}

\begin{proof} Let $\Omega'$ be a domain (occupied by a solid) whose
	boundary is $\Gamma$. We apply the Gauss-Green theorem to $\Omega'$
	and get 
	\begin{align*}
	\int_{\Gamma}\vec{n}\, dS & =-\int_{\Gamma}1\left(-\vec{n}\right)\, dS=-\int_{\Omega'}(\nabla1)\, dx=0\quad\text{and}\\
	\int_{\Gamma}\vec{J}\, dS & =-\int_{\Gamma}(\vec{x}-\vec{c})\times\left(-\vec{n}\right)\, dS=\int_{\Omega'}\nabla\times(\vec{x}-\vec{c})\, dx=0.\qedhere
	\end{align*}
\end{proof}

\subsection{Orthogonality and stability}
For a given triple $(\vec{U}^{*},\vec{v}^{*},\vec{\omega}^{*})$
with a fluid vector field $\vec{U}^{*}:\Omega\to\mathbb{R}^{d}$ and
a solid linear velocity $\vec{v}^{*}\in\mathbb{R}^{d}$ and an angular
velocity $\vec{\omega}^{*}\in\mathbb{R}^{d},$ the kinetic energy
$E^{*}$ is given as 
\[
E^{*}=\intop_{\Omega}\frac{1}{2}\rho\,\abs{\vec{U}^{*}}^{2}\, dx+\frac{1}{2}m\abs{\vec{v}^{*}}^{2}+\frac{1}{2}\vec{\omega}^{*}\cdot\mathbb{I}\vec{\omega}^{*}.
\]
Let $(\vec{U},\vec{v},\vec{\omega})$ and $p$ be as in the decomposition
\eqref{eq-main-decomposition} of $(\vec{U}^{*},\vec{v}^{*},\vec{\omega}^{*})$
and let $E$ the energy induced by the triple $(\vec{U},\vec{v},\vec{\omega})$,
which is given as 
\[
E=\intop_{\Omega}\frac{1}{2}\rho\,\abs{\vec{U}}^{2}\, dx+\frac{1}{2}m\abs{\vec{v}}^{2}+\frac{1}{2}\vec{\omega}\cdot\mathbb{I}\vec{\omega}.
\]
Now, we estimate the energies $E^{*}$ and $E$. Applying the decomposition
\eqref{eq-main-decomposition}, we have 
\begin{align*}
E^{*} & =\int_{\Omega}\frac{1}{2}\rho\Abs{\vec{U}+\rho^{-1}\nabla p}^{2}\, dx+\frac{1}{2}m\Abs{\vec{v}-\frac{1}{m}\int_{\Gamma}p\vec{n}\, dS}^{2}\\
& \qquad\qquad+\frac{1}{2}\left(\vec{\omega}-\mathbb{I}^{-1}\int_{\Gamma}p\vec{J}\, dS\right)\cdot\mathbb{I}\left(\vec{\omega}-\mathbb{I}^{-1}\int_{\Gamma}p\vec{J}\, dS\right)\\
& =\int_{\Omega}\frac{1}{2}\rho\abs{\vec{U}}^{2}\, dx+\frac{1}{2}m\abs{\vec{v}}^{2}+\frac{1}{2}\vec{\omega}\cdot\mathbb{I}\vec{\omega}\\
& \qquad\qquad+\int_{\Omega}\vec{U}\cdot\nabla p\, dx-\vec{v}\cdot\int_{\Gamma}p\vec{n}\, dS-\vec{\omega}\cdot\int_{\Gamma}p\vec{J}\, dS\\
& \qquad\qquad+\int_{\Omega}\frac{1}{2\rho}\,\abs{\nabla p}^{2}\, dx+\frac{1}{2m}\Abs{\int_{\Gamma}p\vec{n}\, dS}^{2}+\frac{1}{2}\left(\mathbb{I}^{-1}\int_{\Gamma}p\vec{J}\, dS\right)\cdot\left(\int_{\Gamma}p\vec{J}\, dS\right)\\[6pt]
& \ge\int_{\Omega}\frac{1}{2}\rho\abs{\vec{U}}^{2}\, dx+\frac{1}{2}m\abs{\vec{v}}^{2}+\frac{1}{2}\vec{\omega}\cdot\mathbb{I}\vec{\omega}=E.
\end{align*}
In the above, we used the symmetric positive definiteness of $\bI$
and the incompressible and the non-penetration conditions \eqref{eq-boundary-Cdns}
which imply 
\begin{multline}
\int_{\Omega}\vec{U}\cdot\nabla p\, dx-\vec{v}\cdot\int_{\Gamma}p\vec{n}\, dS-\vec{\omega}\cdot\int_{\Gamma}p\vec{J}\, dS=-\int_{\Omega}\left(\nabla\cdot\vec{U}\right)p\, dx\\
+\int_{\Gamma'}p\vec{U}\cdot\vec{n}\, dS+\int_{\Gamma}p\vec{U}\cdot\vec{n}\, dS-\int_{\Gamma}p\vec{v}\cdot\vec{n}\, dS-\int_{\Gamma}p\omega\cdot\vec{J}\, dS=0.\label{eq-decomposition-orthogonality}
\end{multline}

Equation \eqref{eq-decomposition-orthogonality} implies that the decomposition
is orthogonal with respect to the inner product $\langle\cdot,\cdot\rangle_{E}$
on $L^{2}(\Omega)^{d}\times\mathbb{R}^{d}\times\mathbb{R}^{d}$ defined
by 
\begin{multline}
\bigip{(\vec{U}^{(1)},\vec{v}^{(1)},\vec{\omega}^{(1)}),(\vec{U}^{(2)},\vec{v}^{(2)},\vec{\omega}^{(2)})}_{E}\\
:=\int_{\Omega}\frac{1}{2}\rho\,\vec{U}^{(1)}\cdot\vec{U}^{(2)}\, dx+\frac{1}{2}m\vec{v}^{(1)}\cdot\vec{v}^{(2)}+\frac{1}{2}\vec{\omega}^{(1)}\cdot\mathbb{I}\vec{\omega}^{(2)}.\label{eq-Def-InnerPdt}
\end{multline}
We can write the energy $E$ of a triple $(\vec{U},\vec{v},\vec{\omega})$
in terms of the inner product as 
\[
E=\int_{\Omega}\frac{1}{2}\rho\,\abs{\vec{U}}^{2}\, dx+\frac{1}{2}m\abs{\vec{v}}^{2}+\frac{1}{2}\vec{\omega}\cdot\mathbb{I}\vec{\omega}=\bigip{(\vec{U},\vec{v},\vec{\omega}),(\vec{U},\vec{v},\vec{\omega})}_{E}.
\]

Consequently, we have shown that the projection $
(\vec{U},\vec{v},\vec{\omega})=\mathcal{P} \left( \vec{U}^{*},\vec{v}^{*},\vec{\omega}^{*} \right),
$ is stable as the kinetic energy $E$ of the projection does not increase comparing with the energy $E^*$ of the state variable. In summary, we have the following:

\begin{theorem} \label{thm-energyEstimate} For a given triple $(\vec{U}^{*},\vec{v}^{*},\vec{\omega}^{*}),$
	let the triple $(\vec{U},\vec{v},\vec{\omega})$ be given by 
	\begin{equation}
	\Big(\vec{U}^{*},\vec{v}^{*},\vec{\omega}^{*}\Big)=\Big(\vec{U},\vec{v},\vec{\omega}\Big)+\left(\frac{1}{\rho}\nabla p,-\frac{1}{m}\int_{\Gamma}p\vec{n}\, dS,-\mathbb{I}^{-1}\int_{\Gamma}p\vec{J}\, dS\right)\label{eq-decomposition2}
	\end{equation}
	for some scalar function $p$. If $(\vec{U},\vec{v},\vec{\omega})$
	satisfies the conditions 
	\[
	\left\{ \begin{aligned}\nabla\cdot\vec{U} & =0\quad\text{in }\;\Omega\\
	\vec{U}\cdot\vec{n} & =\vec{v}\cdot\vec{n}+\vec{\omega}\cdot\vec{J}\quad\text{on }\;\Gamma,\\
	\vec{U}\cdot\vec{n} & =0\quad\text{on }\;\Gamma'.
	\end{aligned}
	\right.
	\]
	then the decomposition \eqref{eq-decomposition2} is orthogonal
	with respect to the inner product \eqref{eq-Def-InnerPdt}. Furthermore,
	the kinetic energy decreases in the sense $E^{*}\ge E$. \end{theorem}

\section{Poisson equation with nonlocal Robin boundary condition}\label{sec3}
Taking into account all relations in terms of a scalar field $p$,
the decomposition \eqref{eq-main-decomposition} is fulfilled by solving 
the Poisson equation with nonlocal Robin boundary condition: 
\begin{equation}
\left\{ \begin{aligned}-\nabla\cdot\left(\frac{1}{\rho}\,\nabla p\right) & =-\nabla\cdot\vec{U}^{*}\;\text{ in }\;\Omega,\\
\frac{1}{\rho}\frac{\partial p}{\partial n}+\vec{n}\cdot\frac{1}{m}{\displaystyle \int_{\Gamma}p\vec{n}\, dS+\vec{J}\cdot\mathbb{I}^{-1}{\displaystyle \int_{\Gamma}p\vec{J}\, dS}} & =\vec{U}^{*}\cdot\vec{n}-\vec{v}^{*}\cdot\vec{n}-\vec{\omega}^{*}\cdot\vec{J}\;\text{ on }\;\Gamma,\\
\frac{1}{\rho}\frac{\partial p}{\partial n} & =\vec{U}^{*}\cdot\vec{n}\;\text{ on }\;\Gamma'.
\end{aligned}
\right.\label{eq-governingEq-FSI}
\end{equation}
In this section, we estimate the existence, uniqueness, and regularity of a solution $p$ to the Poisson equation.

\subsection{Existence and uniqueness}
We will show the existence of a pressure $p$ appearing in the decomposition
\eqref{eq-main-decomposition} - \eqref{eq-boundary-Cdns} as a weak
solution of a Poisson problem with a nonlocal Robin boundary condition; once
$p$ is obtained, we compute $(\vec{U},\vec{v},\vec{\omega})$ by
the formula \eqref{eq-main-decomposition}.

For the sake of generality and to address regularity issue, we consider
the following more general setting: Let $\Omega\subset\bR^{d}$ be
a bounded domain with boundary $\partial\Omega$ decomposed into two
disjoint parts $\partial\Omega=\Gamma\cup\Gamma'$. For $f\in L^{2}(\Omega)$,
$\vec{F}\in L^{2}(\Omega;\bR^{d})$, and $g\in L^{2}(\partial\Omega)=L^{2}(\Gamma\cup\Gamma')$,
we consider the following problem 
\begin{gather}
-\nabla\cdot(\mathbb{A}\nabla u)=f-\nabla\cdot\vec{F}\;\text{ in }\;\Omega,\label{eq:2-05ri}\\
\mathbb{A}\nabla u\cdot\vec{n}=\vec{F}\cdot\vec{n}+g\;\text{ on }\;\Gamma',\label{eq:2.05ri}\\
\mathbb{A}\nabla u\cdot\vec{n}+\left[\alpha\int_{\Gamma}u\vec{n}\, dS+\left(\mathbb{B}\int_{\Gamma}u\vec{r}\times\vec{n}\, dS\right)\times\vec{r}\right]\cdot\vec{n}=\vec{F}\cdot\vec{n}+g\;\text{ on }\;\Gamma.\label{eq:2.06ri}
\end{gather}
Here, $\mathbb{A}$ is a symmetric $d\times d$ matrix valued measurable
function on $\overline{\Omega}$ satisfying the uniform ellipticity
condition; i.e., there are constants $0<\lambda\le\Lambda<+\infty$
such that 
\begin{equation}
\lambda\abs{\vec{\xi}}^{2}\le\mathbb{A}(x)\vec{\xi}\cdot\vec{\xi}\le\Lambda\abs{\vec{\xi}}^{2},\quad\forall x\in\overline{\Omega},\;\;\forall\vec{\xi}\in\bR^{d},\label{eq:2-ellipticity-ri}
\end{equation}
$\alpha$ is a nonnegative constant, $\mathbb{B}$ is a symmetric
nonnegative definite $d\times d$ constant matrix, and $\vec{r}:\Gamma\to\bR^{d}$
is a bounded vector field satisfying $\int_{\Gamma}\vec{r}\times\vec{n}\, dS=0$.

\begin{theorem} \label{thm-Pexistence} Assume the above conditions.
	Then the problem \eqref{eq:2-05ri} -- \eqref{eq:2.06ri} has a weak
	solution $u\in H^{1}(\Omega)$ if and only if $f$ and $g$ satisfy
	the compatibility condition 
	\begin{equation}
	\int_{\Omega}f\, dx+\int_{\partial\Omega}g\, dS=0.\label{eq:compatibility-ri}
	\end{equation}
	Moreover, such a weak solution $u$ is unique modulo an additive constant.
	In particular, one may assume that $\int_{\Omega}u=0$. \end{theorem}

\begin{proof} Let us first formulate the problem \eqref{eq:2-05ri}
	-- \eqref{eq:2.06ri} in a weak setting. We multiply \eqref{eq:2-05ri}
	by a function $v\in H^{1}(\Omega)$ and integrate by parts on $\Omega$
	to obtain 
	\[
	\int_{\Omega}\mathbb{A}\nabla u\cdot\nabla v\, dx=\int_{\partial\Omega}(\mathbb{A}\nabla u\cdot\vec{n}-\vec{F}\cdot\vec{n})v\, dS+\int_{\Omega}fv\, dx+\int_{\Omega}\vec{F}\cdot\nabla v\, dx.
	\]
	The boundary conditions \eqref{eq:2.05ri}, \eqref{eq:2.06ri}, and
	the vector identity $\vec{a}\cdot(\vec{b}\times\vec{c})=(\vec{a}\times\vec{b})\cdot\vec{c}$,
	yield 
	\begin{multline*}
	\int_{\partial\Omega}(\mathbb{A}\nabla u\cdot\vec{n}-\vec{F}\cdot\vec{n})v\, dS=\int_{\partial\Omega}gv\, dS-\alpha\left(\int_{\Gamma}u\vec{n}\, dS\right)\cdot\left(\int_{\Gamma}v\vec{n}\, dS\right)\\
	-\left(\mathbb{B}\int_{\Gamma}u\vec{r}\times\vec{n}\, dS\right)\cdot\left(\int_{\Gamma}v\vec{r}\times\vec{n}\, dS\right).
	\end{multline*}
	Combining together, we have 
	\begin{equation}
	\mathscr{B}[u,v]=\int_{\Omega}fv\, dx+\int_{\Omega}\vec{F}\cdot\nabla v\, dx+\int_{\partial\Omega}gv\, dS.\label{eq:ws-ri}
	\end{equation}
	where we set 
	\begin{multline}
	\mathscr{B}[u,v]:=\int_{\Omega}\mathbb{A}\nabla u\cdot\nabla v\, dx+\alpha\left(\int_{\Gamma}u\vec{n}\, dS\right)\cdot\left(\int_{\Gamma}v\vec{n}\, dS\right)\\
	+\mathbb{B}\left(\int_{\Gamma}u\vec{r}\times\vec{n}\, dS\right)\cdot\left(\int_{\Gamma}v\vec{r}\times\vec{n}\, dS\right).\label{eq:BL}
	\end{multline}
	We shall therefore say that a function $u\in H^{1}(\Omega)$ is a
	weak solution of the problem \eqref{eq:2-05ri} -- \eqref{eq:2.06ri}
	if it satisfies the identity \eqref{eq:ws-ri} for all $v\in H^{1}(\Omega)$.
	
	By Lemma~\ref{lemma-zero-nj} and the assumption $\int_{\Gamma}\vec{r}\times\vec{n}\, dS=0$,
	we have $\mathscr{B}[u,1]=0$ for any $u\in H^{1}(\Omega)$, and thus
	any weak solution of the problem \eqref{eq:2-05ri} -- \eqref{eq:2.06ri}
	should satisfy the compatibility condition \eqref{eq:compatibility-ri}.
	
	Next, we apply the Lax-Milgram theorem to show the existence of a
	weak solution. By the trace theorem, it is clear that $\mathscr{B}[\cdot,\cdot]$
	is bounded; i.e., there is a constant $C$ such that 
	\[
	\abs{\mathscr{B}[u,v]}\le C\norm{u}_{H^{1}(\Omega)}\norm{v}_{H^{1}(\Omega)},\quad\forall u,\, v\in H^{1}(\Omega).
	\]
	However, $\mathscr{B}[\cdot,\cdot]$ is not coercive on $H^{1}$.
	To remedy this, let us introduce a subspace $H_{*}^{1}(\Omega)$ of
	$H^{1}(\Omega)$ defined by 
	\[
	H_{*}^{1}(\Omega):=\{u\in H^{1}(\Omega):\int_{\Omega}u=0\}.
	\]
	Clearly, $H_{*}^{1}(\Omega)$ is a closed subspace of $H^{1}(\Omega)$
	and thus, $H_{*}^{1}(\Omega)$ itself is a Hilbert space. Let us recall
	that Poincar\'e's inequality: 
	\[
	\norm{u-(u)_{\Omega}}_{L^{2}(\Omega)}\le C\norm{\nabla u}_{L^{2}(\Omega)};\quad(u)_{\Omega}=\text{the average of \ensuremath{u} over \ensuremath{\Omega}}.
	\]
	When restricted to $H_{*}^{1}(\Omega)$, Poincar\'e's inequality and
	the assumption that $\alpha\ge0$ and $\mathbb{B}\ge0$ imply that
	$\mathscr{B}[\cdot,\cdot]$ is coercive on $H_{*}^{1}(\Omega)$; i.e.,
	there is a constant $c>0$ such that 
	\begin{equation}
	\mathscr{B}[u,u]\ge c\norm{u}_{H^{1}(\Omega)}^{2},\quad\forall u\in H_{*}^{1}(\Omega).\label{eq-B-bounded-below}
	\end{equation}
	As a matter of fact, for any $u\in H^{1}(\Omega)$, we have 
	\begin{multline}
	\mathscr{B}[u,u]=\int_{\Omega}\mathbb{A}\nabla u\cdot\nabla u\, dx+\alpha\Abs{\int_{\Gamma}u\vec{n}\, dS}^{2}\\
	+\mathbb{B}\left(\int_{\Gamma}u\vec{r}\times\vec{n}\, dS\right)\cdot\left(\int_{\Gamma}u\vec{r}\times\vec{n}\, dS\right)\ge\lambda\int_{\Omega}\abs{\nabla u}^{2}\, dx,\label{eq:02.11ri}
	\end{multline}
	and Poincar\'e's inequality yields \eqref{eq-B-bounded-below}. Also,
	the trace theorem verifies that the linear functional $F$ on $H^{1}(\Omega)$
	given by 
	\begin{equation}
	F(v):=\int_{\Omega}fv\, dx+\int_{\Omega}\vec{F}\cdot\nabla v\, dx+\int_{\partial\Omega}gv\, dS\label{eq:LF}
	\end{equation}
	is bounded, and thus it is a bounded linear functional on $H_{*}^{1}(\Omega)$
	as well.
	
	Therefore, the Lax-Milgram theorem implies that there exists a unique
	$u\in H_{*}^{1}(\Omega)$ such that 
	\[
	\mathscr{B}[u,v]=F(v)=\int_{\Omega}fv\, dx+\int_{\Omega}\vec{F}\cdot\nabla v\, dx+\int_{\partial\Omega}gv\, dS,\quad\forall v\in H_{*}^{1}(\Omega).
	\]
	We now show that $u$ satisfies the identity \eqref{eq:ws-ri} for
	any $v\in H^{1}(\Omega)$ so that $u$ is indeed a weak solution of
	the problem \eqref{eq:2-05ri} -- \eqref{eq:2.06ri}. Note that any
	$v\in H^{1}(\Omega)$ can be written as $v=\tilde{v}+c$, where $\tilde{v}\in H_{*}^{1}(\Omega)$
	and $c$ is a constant. Since $\mathscr{B}[u,c]=0$ by Lemma~\ref{lemma-zero-nj}
	and the assumption $\int_{\Gamma}\vec{r}\times\vec{n}\, dS=0$, and
	$F(c)=0$ by the compatibility condition \eqref{eq:compatibility-ri},
	the identity \eqref{eq:ws-ri} holds for $u$. Finally, we show that
	a weak solution $u$ is unique modulo an additive constant. Thanks
	to linearity, it is enough to show that any $u\in H^{1}(\Omega)$
	satisfying the identity 
	\[
	\mathscr{B}[u,v]=0,\quad\forall v\in H^{1}(\Omega)
	\]
	must be a constant. By taking $v=u$ in the above, we get by \eqref{eq:02.11ri}
	that 
	\[
	0=\mathscr{B}[u,u]\ge\lambda\int_{\Omega}\abs{\nabla u}^{2}\, dx,
	\]
	which obviously implies that $u$ is a constant. \end{proof}

\begin{remark} In fact, we can consider the case when $\Omega\subset\bR^{3}$
	is an (unbounded) exterior domain and $\Gamma'=\emptyset$. In this
	case, we also get a similar result except that we do not need to impose
	the compatibility condition. Here, we briefly describe the proof.
	Instead of working in $H^{1}(\Omega)=W^{1,2}(\Omega)$, we use a different
	function space $Y^{1,2}(\Omega)$, which is defined as the family
	of all weakly differentiable functions $u\in L^{6}(\Omega)$, whose
	weak derivatives are functions in $L^{2}(\Omega)$. The space $Y^{1,2}(\Omega)$
	is endowed with the norm 
	\[
	\norm{u}_{Y^{1,2}(\Omega)}:=\norm{u}_{L^{6}(\Omega)}+\norm{\nabla u}_{L^{2}(\Omega)}.
	\]
	It is known that if $\Omega=\bR^{3}\setminus\overline{D}$ and $D$
	is a bounded Lipschitz domain, then we have the following Sobolev
	inequality (\cite{Chen:Williams:Zhao}): 
	\begin{equation}
	\norm{u}_{L^{6}(\Omega)}\le C\norm{\nabla u}_{L^{2}(\Omega)},\quad\forall u\in Y^{1,2}(\Omega).\label{eq:ysob}
	\end{equation}
	Therefore, in this case, $Y^{1,2}(\Omega)$ becomes a Hilbert space
	under the inner product 
	\[
	\ip{u,v}_{Y^{1,2}(\Omega)}=\int_{\Omega}\nabla u\cdot\nabla v\, dx.
	\]
	Also, we have the trace inequality 
	\[
	\norm{u}_{L^{2}(\partial\Omega)}\le C\norm{\nabla u}_{L^{2}(\Omega)},\quad\forall u\in Y^{1,2}(\Omega).
	\]
	To see this, assume $\overline{D}\subset B(x_{0},r)$ and fix $\eta\in C_{c}^{\infty}(B(x_{0},2r))$
	be such that $0\le\eta\le1$, $\eta=1$ on $B(x_{0},r)$, and $\abs{\nabla\eta}\le2/r$;
	one may take $r=\diam D$. Denote $\Omega_{r}=\Omega\cap B(x_{0},r)$.
	The usual trace inequality implies 
	\[
	\norm{\eta u}_{L^{2}(\partial\Omega_{2r})}\le C\norm{\eta u}_{W^{1,2}(\Omega_{2r})}.
	\]
	Note that $\norm{\eta u}_{L^{2}(\partial\Omega_{2r})}=\norm{u}_{L^{2}(\partial\Omega)}$
	and $\norm{\eta u}_{W^{1,2}(\Omega_{2r})}\le C\norm{\nabla u}_{L^{2}(\Omega)}$
	by the Sobolev inequality \eqref{eq:ysob} and H\"older's inequality.
	More precisely, we have 
	\[
	\int_{\Omega_{2r}}\abs{\eta u}^{2}\, dx\le Cr^{2}\left(\int_{\Omega_{2r}}\abs{u}^{6}\, dx\right)^{\frac{1}{3}}\le Cr^{2}\left(\int_{\Omega}\abs{u}^{6}\, dx\right)^{\frac{1}{3}}\le Cr^{2}\int_{\Omega}\abs{\nabla u}^{2}\, dx
	\]
	and similarly we get 
	\begin{align*}
	\int_{\Omega_{2r}}\abs{\nabla(\eta u)}^{2}\, dx & \le2\int_{\Omega_{2r}}\eta^{2}\abs{\nabla u}^{2}\, dx+2\int_{\Omega_{2r}}\abs{\nabla\eta}^{2}\abs{u}^{2}\, dx\\
	& \le2\int_{\Omega_{2r}}\abs{\nabla u}^{2}\, dx+C\int_{\Omega}\abs{\nabla u}^{2}\, dx\le C\int_{\Omega}\abs{\nabla u}^{2}\, dx.
	\end{align*}
	Therefore, the bilinear form \eqref{eq:BL} is bounded and coercive
	on the Hilbert space $Y^{1,2}(\Omega)$. Also, if $\vec{F}\in L^{2}(\Omega)$,
	$f\in L^{6/5}(\Omega)$, and $g\in L^{2}(\partial\Omega)$, then the
	linear functional $F(\cdot)$ in \eqref{eq:LF} is bounded on $Y^{1,2}(\Omega)$,
	and thus Lax-Milgram theorem implies the existence of a weak solution.
\end{remark}

\begin{remark} We may also replace the boundary condition on $\Gamma'$
	to non-slip condition $\vec{U}=0$. \end{remark}

\subsection{Regularity}
In this section, we study regularity of a weak solution
$u\in H^{1}(\Omega)$ of the problem \eqref{eq:2-05ri} -- \eqref{eq:2.06ri}.
As in the previous section, we assume that $\mathbb{A}$ satisfies
the uniform ellipticity condition \eqref{eq:2-ellipticity-ri}, $\alpha\ge0$,
$\mathbb{B}\ge0$, and $\int_{\Gamma}\vec{r}\times\vec{n}\, dS=0$.

Suppose $f$ and $g$ satisfy the compatibility condition \eqref{eq:compatibility-ri}
and let $u\in H^{1}(\Omega)$ be a weak solution of the problem \eqref{eq:2-05ri}
-- \eqref{eq:2.06ri}. By Theorem~\ref{thm-Pexistence}, such a weak
solution is unique up to an additive constant, and thus by Lemma~\ref{lemma-zero-nj}
and the assumption $\int_{\Gamma}\vec{r}\times\vec{n}\, dS=0$, we
see that the (constant) vectors 
\[
\vec{a}:=\int_{\Gamma}u\vec{n}\, dS\quad\text{and}\quad\vec{b}:=\int_{\Gamma}u\vec{r}\times\vec{n}\, dS,
\]
are uniquely determined (independent of an additive constant). Therefore,
a weak solution $u\in H^{1}(\Omega)$ of the problem \eqref{eq:2-05ri}
-- \eqref{eq:2.06ri} is also a weak solution of 
\begin{equation}
\left\{ \begin{aligned}-\nabla\cdot(\mathbb{A}\nabla u) & =f-\nabla\cdot\vec{F}\;\text{ in }\;\Omega,\\
\mathbb{A}\nabla u\cdot\vec{n} & =\vec{F}\cdot\vec{n}+g\;\text{ on }\;\Gamma',\\
\mathbb{A}\nabla u\cdot\vec{n} & =\vec{F}\cdot\vec{n}+g-h\;\text{ on }\;\Gamma,
\end{aligned}
\right.\label{eq3.01ri}
\end{equation}
where we set 
\begin{equation}
h:=(\alpha\vec{a}+\mathbb{B}\vec{b}\times\vec{r})\cdot\vec{n}=\left[\alpha\int_{\partial\Omega}u\vec{n}\, dS+\left(\mathbb{B}\int_{\partial\Omega}u\vec{r}\times\vec{n}\, dS\right)\times\vec{r}\right]\cdot\vec{n}.\label{eq3.02ri}
\end{equation}
It is easy to check that 
\[
h\in L^{\infty}(\partial\Omega),\quad\int_{\partial\Omega}h\, dS=0,\quad\text{and \ensuremath{h} is as regular as \ensuremath{\vec{r}\times\vec{n}}}.
\]
In particular, if $\vec{r}\in C^{k}(\Gamma;\,\bR^{d})$ and $\Gamma$
is of class $C^{k+1}$, then $h\in C^{k}(\Gamma)$, etc. Therefore,
we have the following results from the standard elliptic regularity
theory and we omit the proofs because they are straightforwardly derived;
see, e.g., \cite{GT,Evans}.

\begin{theorem}[Interior $H^k$-regularity]\label{theorem-interior-regularity}
	Assume $\mathbb{A}\in C^{k+1}(\Omega)$, $f\in H^{k}(\Omega)$, and
	$\vec{F}\in H^{k+1}(\Omega)$, where $k\ge0$ is an integer. If $u\in H^{1}(\Omega)$
	is a weak solution of 
	\[
	-\nabla\cdot(\mathbb{A}\nabla u)=f-\nabla\cdot\vec{F}\;\text{ in }\;\Omega,
	\]
	then $u\in H_{loc}^{k+2}(\Omega)$ and for $\Omega'\subset\subset\Omega$,
	we have the estimate 
	\[
	\norm{u}_{H^{k+2}(\Omega')}\le C\left(\norm{f}_{H^{k}(\Omega)}+\norm{\vec{F}}_{H^{k+1}(\Omega)}+\norm{u}_{L^{2}(\Omega)}\right).
	\]
	where the constant $C$ depends only on $k,\Omega,\Omega'$, and $\mathbb{A}$.
\end{theorem}

\begin{theorem}[Interior $C^{k+2,\alpha}$-regularity] Assume $\mathbb{A}\in C_{loc}^{k+1,\alpha}(\Omega)$,
	$f\in C_{loc}^{k,\alpha}(\Omega)$, and $\vec{F}\in C_{loc}^{k+1,\alpha}(\Omega)$,
	where $k\ge0$ is an integer. If $u\in H^{1}(\Omega)$ is a weak solution
	of 
	\[
	-\nabla\cdot(\mathbb{A}\nabla u)=f-\nabla\cdot\vec{F}\;\text{ in }\;\Omega,
	\]
	then $u\in C_{loc}^{k+2,\alpha}(\Omega)$ and for $\Omega''\subset\subset\Omega'\subset\subset\Omega$,
	we have the estimate 
	\[
	\norm{u}_{C^{k+2,\alpha}(\Omega'')}\le C\left(\norm{f}_{C^{k,\alpha}(\Omega')}+\norm{\vec{F}}_{C^{k+1,\alpha}(\Omega')}+\norm{u}_{L^{2}(\Omega)}\right).
	\]
	where the constant $C$ depends only on $k,\Omega,\Omega'$, and $\mathbb{A}$.
\end{theorem}

\begin{theorem}[Global $H^k$-regularity]\label{theorem-boundary-requality-higher}
	Assume $\partial\Omega$ is $C^{k+2}$, $\mathbb{A}\in C^{k+1}(\overline{\Omega})$,
	$f\in H^{k}(\Omega)$, $\vec{F}\in H^{k+1}(\Omega)$, $g\in H^{k+\frac{1}{2}}(\partial\Omega)$,
	and $\vec{r}\in H^{k+\frac{1}{2}}(\Gamma)$, where $k\ge0$ is an
	integer. If $u\in H^{1}(\Omega)$ is a weak solution of the problem
	\eqref{eq:2-05ri} -- \eqref{eq:2.06ri}, then $u\in H^{k+2}(\Omega)$.
	In particular, if we fix $u$ so that $\int_{\Omega}u\, dx=0$, then
	we have the estimate 
	\[
	\norm{u}_{H^{k+2}(\Omega)}\le C\left(\norm{f}_{H^{k}(\Omega)}+\norm{\vec{F}}_{H^{k+1}(\Omega)}+\norm{g}_{H^{k+\frac{1}{2}}(\partial\Omega)}+\norm{\vec{r}}_{H^{k+\frac{1}{2}}(\partial\Omega)}\right),
	\]
	where the constant $C$ depends only on $\Omega$ and the coefficients
	$\mathbb{A}$, $\alpha$, $\mathbb{B}$. \end{theorem}

\begin{theorem}[Global $C^{k+2, \alpha}$-regularity]\label{theorem-boundary-regularity-continous} Let $k\ge0$
	be an integer and $\alpha\in(0,1)$. Assume $\partial\Omega$ is $C^{k+2,\alpha}$,
	$\mathbb{A}\in C^{k+1,\alpha}(\Omega)$, $f\in C^{k,\alpha}(\Omega)$,
	$\vec{F}\in C^{k+1,\alpha}(\Omega)$, $g\in C^{k+1,\alpha}(\partial\Omega)$,
	and $\vec{r}\in C^{k+1,\alpha}(\Gamma)$. If $u\in H^{1}(\Omega)$
	be a weak solution of the problem \eqref{eq:2-05ri} -- \eqref{eq:2.06ri},
	then $u\in C^{k+2,\alpha}(\Omega)$. In particular, if we fix $u$
	so that $\int_{\Omega}u\, dx=0$, then we have the estimate 
	\[
	\norm{u}_{C^{k+2,\alpha}(\Omega)}\le C\left(\norm{f}_{C^{k,\alpha}(\Omega)}+\norm{\vec{F}}_{C^{k+1,\alpha}(\Omega)}+\norm{g}_{C^{k+1,\alpha}(\partial\Omega)}+\norm{\vec{r}}_{C^{k+1,\alpha}(\partial\Omega)}\right),
	\]
	where the constant $C$ depends only on $\Omega$ and the coefficients
	$\mathbb{A}$, $\alpha$, $\mathbb{B}$. \end{theorem}




\subsection{Monolithic decomposition}
Now, we are ready to show the existence and uniqueness of the augmented Hodge decomposition with fluid-solid interaction mentioned in Section 2. As shown in the section, the decomposition is achieved monolithically as soon as we find the scalar field $p$. 

\begin{theorem}[Hodge decomposition with fluid-solid interaction]
	\label{thm-HHD-FSI} 
	Let $k\ge0$ be an integer and $\alpha\in(0,1)$. Assume $\partial\Omega$ is $C^{k+2}$ (resp., $C^{k+2,\alpha}$).
	Then, for any vector field $\vec{U}^{*}\in H^{k+1}(\Omega)$ (resp., $C^{k+1,\alpha}(\Omega)$)
	and any linear and angular velocities $\vec{v}^{*},\vec{\omega}^{*}\in\mathbb{R}^{d}$,
	the triple $(\vec{U}^{*},\vec{v}^{*},\vec{\omega}^{*})$ are uniquely
	decomposed as 
	\begin{equation}
	\left\{ \begin{aligned}\rho\vec{U}^{*} & =\rho\vec{U}+\nabla p\quad\text{in }\;\Omega\\
	m\vec{v}^{*} & =m\vec{v}-{\displaystyle \int_{\Gamma}p\vec{n}\, dS}\\
	\mathbb{I}\vec{\omega}^{*} & =\mathbb{I}\vec{\omega}-{\displaystyle \int_{\Gamma}p\vec{J}\, dS}
	\end{aligned}
	\right.\label{eq-FSI-decomposition}
	\end{equation}
	with a vector field $\vec{U}\in H^{k+1}(\Omega)$ (resp., $C^{k+1,\alpha}(\Omega)$), 
	$p\in H^{k+2}(\Omega)$ (resp., $C^{k+2,\alpha}(\Omega)$), and vectors $\vec{v},\vec{\omega}\in\mathbb{R}^{d}$
	that satisfy the incompressible condition and the non-penetration
	condition 
	\[
	\left\{ \begin{aligned}\nabla\cdot\vec{U} & =0\quad\text{in }\;\Omega,\\
	\vec{U}\cdot\vec{n} & =\vec{v}\cdot\vec{n}+\vec{\omega}\cdot\vec{J}\quad\text{on }\;\Gamma,\\
	\vec{U}\cdot\vec{n} & =0\quad\text{on }\;\Gamma'.
	\end{aligned}
	\right.
	\]
\end{theorem}

\begin{proof} The decomposition \eqref{eq-FSI-decomposition} yields
	the following problem for the scalar field $p$: 
	\begin{equation}
	\left\{ \begin{aligned}-\nabla\cdot\left(\frac{1}{\rho}\,\nabla p\right) & =-\nabla\cdot\vec{U}^{*}\;\text{ in }\;\Omega,\\
	\frac{1}{\rho}\frac{\partial p}{\partial n}+\vec{n}\cdot\frac{1}{m}{\displaystyle \int_{\Gamma}p\vec{n}\, dS+\vec{J}\cdot\mathbb{I}^{-1}{\displaystyle \int_{\Gamma}p\vec{J}\, dS}} & =\vec{U}^{*}\cdot\vec{n}-\vec{v}^{*}\cdot\vec{n}-\vec{\omega}^{*}\cdot\vec{J}\;\text{ on }\;\Gamma,\\
	\frac{1}{\rho}\frac{\partial p}{\partial n} & =\vec{U}^{*}\cdot\vec{n}\;\text{ on }\;\Gamma'.
	\end{aligned}
	\right.\label{eq-governingEq-FSI2}
	\end{equation}
	Set $\vec{F}=\vec{U}^{*}$, $f=0$, and $g=-(\vec{v}^{*}\cdot\vec{n}+\vec{\omega}^{*}\cdot\vec{J})\chi_{\Gamma}$.
	Applying Lemma~\ref{lemma-zero-nj} for the compatibility condition
	and setting $\vec{r}:=\vec{x}-\vec{c}$ so that $\vec{J}=\vec{r}\times\vec{n}$,
	Theorems~\ref{thm-Pexistence} and \ref{theorem-boundary-requality-higher} (resp., Theorem \ref{theorem-boundary-regularity-continous}) imply that there exists a solution
	$p\in \in H^{k+2}(\Omega)\,~(resp., ~C^{k+2,\alpha}(\Omega))$ of the problem \eqref{eq-governingEq-FSI2}.
	Then $\vec{U},\vec{v}$, and $\vec{\omega}$ are defined as 
	\[
	\vec{U}=\vec{U}^{*}-\frac{1}{\rho}\,\nabla p,\quad\vec{v}=\vec{v}^{*}+\frac{1}{m}\int_{\Gamma}p\vec{n}\, dS,\quad\text{and}\quad\vec{\omega}=\vec{\omega}^{*}+\mathbb{I}^{-1}\left(\int_{\Gamma}p\vec{J}\, dS\right).
	\]
	This shows the existence of a triple $(\vec{U},\vec{v},\vec{\omega})$ with $\vec{U}\in H^{k+1}(\Omega)$ (resp., $C^{k+1,\alpha}(\Omega)$)
	for the desired decomposition. Now, we show the uniqueness. To the
	end, by linearity it suffices to show that there exists only a trivial
	triple $(\vec{U},\vec{v},\vec{\omega})=(0,0,0)$ satisfying 
	\begin{equation}
	\rho\vec{U}=-\nabla p\;\text{ in }\;\Omega,\quad m\vec{v}=\int_{\Gamma}p\vec{n}\, dS,\quad\text{and}\quad\mathbb{I}\vec{\omega}=\int_{\Gamma}p\vec{J}\, dS\label{eq-equations-trivial}
	\end{equation}
	with the conditions $\nabla\cdot\vec{U}=0$ in $\Omega$, $\vec{U}\cdot\vec{n}=\vec{v}\cdot\vec{n}+\omega\cdot\vec{J}$
	on $\Gamma$, and $\vec{U}\cdot\vec{n}=0$ on $\Gamma'$. Taking into
	account all relations in terms of $p$, the scalar function $p$ must
	satisfy 
	\[
	\left\{ \begin{aligned}-\nabla\cdot\left(\frac{1}{\rho}\,\nabla p\right) & =0\;\text{ in }\;\Omega,\\
	\frac{1}{\rho}\frac{\partial p}{\partial n}+\vec{n}\cdot\frac{1}{m}{\displaystyle \intop_{\Gamma}p\vec{n}\, dS+\vec{J}\cdot\mathbb{I}^{-1}{\displaystyle \intop_{\Gamma}p\vec{J}\, dS}} & =0\;\text{ on }\;\Gamma,\\
	\frac{1}{\rho}\frac{\partial p}{\partial n} & =0\;\text{ on }\;\Gamma'.
	\end{aligned}
	\right.
	\]
	In light of Theorem \ref{thm-Pexistence}, we have that $\nabla p\equiv0,$ that is, $p$ is constant.
	Therefore, by applying Lemma~\ref{lemma-zero-nj} to \eqref{eq-equations-trivial}, we get that $\vec{U}=0$, $\vec{v}=0$, and $\vec{\omega}=0$. We have thus
	shown that for an input $(\vec{U}^{*},\vec{v}^{*},\vec{\omega}^{*})$ with $\vec{U}^{*}\in H^1(\Omega),$
	there exists a unique $(\vec{U},\vec{v},\vec{\omega})$ satisfying
	the decomposition \eqref{eq-FSI-decomposition}, which proves the
	theorem. \end{proof}

\section{Discretization by Heaviside function}\label{sec4}

\subsection{Heaviside function}

It is more convenient to express boundary conditions given on the
interface $\Gamma=\partial\Omega$ in the entire domain $\mathbb{R}^{d}$.
To do this, we consider the Heaviside function $H(x)=\chi_{\Omega}(x)$,
which equal to $1$ for $x\in\Omega$ and $0$ elsewhere. Then $\nabla H=-\delta_{\Gamma}\vec{n}$,
where $\delta_{\Gamma}$ is the Dirac delta function supported on
$\Gamma$ and $\vec{n}$ the outward normal vector at $\Gamma$. Using
these notations, the boundary conditions \eqref{eq-governingEq-FSI}
of $p$ for the Helmholtz-Hodge decomposition \eqref{eq-FSI-decomposition}
in fluid-solid interaction is represented as 
\begin{multline}
-\nabla\cdot\left(\frac{H}{\rho}\nabla p\right)+\nabla H\cdot\frac{1}{m}\left(\int_{\bR^{d}}p\nabla H\, dx\right)+(\vec{r}\times\nabla H)\cdot\bI^{-1}\left(\int_{\bR^{d}}p(\vec{r}\times\nabla H)\, dx\right)\\
=-\nabla\cdot(H\vec{U}^{*})+\vec{v}^{*}\cdot\nabla H+\vec{\omega}^{*}\cdot(\vec{r}\times\nabla H)\quad\text{in }\;\bR^{d}.\label{eq-Heaviside-governingEQ}
\end{multline}
with $\vec{r}=(\vec{x}-\vec{c}).$ Based on the Heaviside formulation,
we propose a numerical scheme to approximate $p$ in the following
subsection.

\subsection{Discretization based on Heaviside formulation}

In order to propose a numerical scheme for the problem, we introduce
numerical settings. First, we consider the case $d=2.$ Let $h\mathbb{Z}^{2}$
denote the uniform grid in $\mathbb{R}^{2}$ with step size $h$.
For each grid node $(x_{i},y_{j})\in h\mathbb{Z}^{2}$, $C_{ij}$
denotes the rectangular control volume centered at the node, and its
four edges are denoted by $E_{i\pm\frac{1}{2},j}$ and $E_{i,j\pm\frac{1}{2}}$
as follows.

\[
\begin{array}{rcccc}
C_{ij} & := & [x_{i-\frac{1}{2}},x_{i+\frac{1}{2}}] & \times & [y_{j-\frac{1}{2}},y_{j+\frac{1}{2}}]\\
E_{i\pm\frac{1}{2},j} & := & x_{i\pm\frac{1}{2}} & \times & [y_{j-\frac{1}{2}},y_{j+\frac{1}{2}}]\\
E_{i,j\pm\frac{1}{2}} & := & [x_{i-\frac{1}{2}},x_{i+\frac{1}{2}}] & \times & y_{j\pm\frac{1}{2}}
\end{array}
\]

Based on the MAC configuration, we define the node set and the edge
sets.

\begin{definition} $\Omega^{h}:=\Set{(x_{i},y_{j})\in h\mathbb{Z}^{2}:C_{ij}\cap\Omega\ne\emptyset}$
	is the set of nodes whose control volumes intersect the domain. In
	the same way, edge sets are defined as 
	\[
	\begin{aligned}E_{x}^{h}:=\Set{(x_{i+\frac{1}{2}},y_{j}):E_{i+\frac{1}{2},j}\cap\Omega\ne\emptyset},\quad E_{y}^{h}:=\Set{(x_{i},y_{j+\frac{1}{2}}):E_{i,j+\frac{1}{2}}\cap\Omega\ne\emptyset},\end{aligned}
	\]
	and $E^{h}:=E_{x}^{h}\cup E_{y}^{h}$. \end{definition}

Note that whenever $E_{i+\frac{1}{2},j}\cap\Omega\ne\emptyset$, $C_{ij}\cap\Omega\ne\emptyset$
and $C_{i+1,j}\cap\Omega\ne\emptyset$, since $E_{i+\frac{1}{2},j}\subset C_{ij}$
and $E_{i+\frac{1}{2},j}\subset C_{i+1,j}$.

In case when $d=3,$ the settings and related definitions are almost
the same as those for 2 dimensional case. Let $h\mathbb{Z}^{3}$ denote
the uniform grid in $\mathbb{R}^{3}$ with step size $h$. For each
grid node $(x_{i},y_{j},z_{k})\in h\mathbb{Z}^{3}$, $C_{ijk}$ denotes
the hexahedron control volume centered at the node, and its six faces
are denoted by $F_{i\pm\frac{1}{2},j,k}$ and $F_{i,j\pm\frac{1}{2},k}$
and $F_{i,j,k\pm\frac{1}{2}}$ as follows. 
\[
\begin{array}{rcccccc}
C_{ijk} & := & [x_{i-\frac{1}{2}},x_{i+\frac{1}{2}}] & \times & [y_{j-\frac{1}{2}},y_{j+\frac{1}{2}}] & \times & [z_{k-\frac{1}{2}},z_{k+\frac{1}{2}}]\\
F_{i\pm\frac{1}{2},j,k} & := & x_{i\pm\frac{1}{2}} & \times & [y_{j-\frac{1}{2}},y_{j+\frac{1}{2}}] & \times & [z_{k-\frac{1}{2}},z_{k+\frac{1}{2}}]\\
F_{i,j\pm\frac{1}{2},k} & := & [x_{i-\frac{1}{2}},x_{i+\frac{1}{2}}] & \times & y_{j\pm\frac{1}{2}} & \times & [z_{k-\frac{1}{2}},z_{k+\frac{1}{2}}]\\
F_{i,j,k\pm\frac{1}{2}} & := & [x_{i-\frac{1}{2}},x_{i+\frac{1}{2}}] & \times & [y_{j-\frac{1}{2}},y_{j+\frac{1}{2}}] & \times & z_{k\pm\frac{1}{2}}.
\end{array}
\]
Similarly, we define the node set $\Omega^{h}=\Set{(x_{i},y_{j})\in h\mathbb{Z}^{3}:C_{ijk}\cap\Omega\ne\emptyset}$
and the face sets $F_{x}^{h},F_{y}^{h},F_{z}^{h}$, and then $F^{h}=F_{x}^{h}\cup F_{y}^{h}\cup F_{z}^{h}$.

We split the node points into the inside points and the near-boundary
points as 
\begin{align*}
\Omega_{\circ}^{h} & :=\set{(x_{i},y_{j})\in\Omega^{h}:H_{i\pm\frac{1}{2},j}=H_{i,j\pm\frac{1}{2}}=1}\quad\text{for }\; d=2,\\
\Omega_{\circ}^{h} & :=\set{(x_{i},y_{j},z_{k})\in\Omega^{h}:H_{i\pm\frac{1}{2},j,k}=H_{i,j\pm\frac{1}{2},k}=H_{i,j,k\pm\frac{1}{2}}=1}\quad\text{for }\; d=3,
\end{align*}
and $\Omega_{\Gamma}^{h}:=\Omega^{h}\setminus\Omega_{\circ}^{h}.$

Since the arguments to the results given in this section are almost
the same, we consider only the two dimensional case and the results
can be extended easily to three dimension.

By the standard central finite differences, a discrete gradient and
a discrete divergence operators are defined as follows.

\begin{definition}[Discrete gradient and divergence operators]
	By $\dG_{x}$, we denote the central finite differences in the $x$-direction:
	\[
	(\dG_{x}p)_{i+\frac{1}{2},j}=\frac{p_{i+1,j}-p_{i,j}}{h}\quad\text{and}\quad(\dG_{x}u)_{i,j}=\frac{u_{i+\frac{1}{2},j}-u_{i-\frac{1}{2},j}}{h}.
	\]
	Similarly, $\dG_{y}$ denotes the central finite differences in the
	$y$-direction. The discrete gradient and divergence operators, denoted
	by $\dG$ and $\dD$ respectively, are defined as 
	\begin{align*}
	\dG[p_{i,j}]=\left([(\dG_{x}p)_{i+\frac{1}{2},j}],[(\dG_{y}p)_{i,j+\frac{1}{2}}]\right),\\
	\dD\left([u_{i+\frac{1}{2},j}],[v_{i,j+\frac{1}{2}}]\right)=[(\dG_{x}u+\dG_{y}v)_{ij}].
	\end{align*}
\end{definition}

From the definitions of discrete gradient and divergence operators,
we can see that one is the adjoint operator of the other. In other
word, the two discrete operators satisfy the integration by parts.

\begin{lemma}[Discrete integration by parts]\label{lemma-DInt-by-Part}
	Given a vector field and a scalar function 
	\[
	\vec{U}=\left([u_{i+\frac{1}{2},j}],[v_{i,j+\frac{1}{2}}]\right)\quad\text{and}\quad p=[p_{ij}]
	\]
	with support $E^{h}$ and $\Omega^{h}$ respectively, we have 
	\begin{multline*}
	\int(\dD\vec{U})p\, d\Omega^{h}:=\sum_{i,j}(\dD\vec{U})_{ij}\, p_{ij}\, h^{2}\\
	=-\sum_{i,j}\left(u_{i+\frac{1}{2},j}\,\dG_{x}p_{i+\frac{1}{2},j}+v_{i,j+\frac{1}{2}}\,\dG_{y}p_{i,j+\frac{1}{2}}\right)h^{2}=:-\int(\vec{U}\cdot\dG p)\, d\Omega^{h}.
	\end{multline*}
\end{lemma} \begin{proof} Using the definitions of discrete gradient
and divergence operators, it is easy to show this lemma. \end{proof}

The Heaviside functions are defined on the edge set as follows.

\begin{definition}[Heaviside function] For each edge, the Heaviside
	function $H$ is given as 
	\[
	H_{i+\frac{1}{2},j}=\frac{\length(E_{i+\frac{1}{2},j}\cap\Omega)}{\length(E_{i+\frac{1}{2},j})}\quad\text{and}\quad H_{i,j+\frac{1}{2}}=\frac{\length(E_{i,j+\frac{1}{2}}\cap\Omega)}{\length(E_{i,j+\frac{1}{2}})}.
	\]
\end{definition}

Note that $H_{i+\frac{1}{2},j}$, $H_{i,j+\frac{1}{2}}\in[0,1]$ are
equal to $1$ if and only if the edge lies totally inside the domain,
and $0$ if and only if the edge lies completely outside.

For $d=3,$ the Heaviside functions are defined on $F^{h}$ as follows.
\begin{gather*}
H_{i+\frac{1}{2},j,k}=\frac{\area(F_{i+\frac{1}{2},j,k}\cap\Omega)}{\area(F_{i+\frac{1}{2},j,k})},\quad H_{i,j+\frac{1}{2},k}=\frac{\area(F_{i,j+\frac{1}{2},k}\cap\Omega)}{\area(F_{i,j+\frac{1}{2},k})},\\
H_{i,j,k+\frac{1}{2}}=\frac{\area(F_{i,j,k+\frac{1}{2}}\cap\Omega)}{\area(F_{i,j,k+\frac{1}{2}})}.
\end{gather*}

Having defined discrete gradient and divergence operators and the
discrete Heaviside function, we have the following lemma, which is
analogous to Lemma \ref{lemma-zero-nj}.

\begin{lemma} \label{lemma-discrte-nj-zero} The discrete Heaviside
	function $H=[H_{i+\frac{1}{2},j}]\cup[H_{i,j+\frac{1}{2}}]$ satisfies
	the relations 
	\[
	\sum_{i,j}(\dG H)_{ij}=0\quad\text{and}\quad\sum_{i,j}((x_{i},y_{j})-\vec{c})\times(\dG H)_{ij}=0.
	\]
\end{lemma} \begin{proof} The proof is straightforward because every
nonzero element of $H$ appears twice with opposite sign in the summations.
\end{proof}

Now, we are ready to formulate a discretization for Equation~\eqref{eq-governingEq-FSI}
based on the Heaviside representation \eqref{eq-Heaviside-governingEQ}.
Using the settings and the definition of $H$ defined as a scalar
function on $E^{h}$, we derive a numerical scheme for $p$ as 
\begin{multline}
-\dD\left(\frac{H}{\rho}\dG p\right)+\dG H\cdot\frac{1}{m}\left(\sum p\dG Hh^{2}\right)+\vec{J}^{h}\cdot\bI^{-1}\left(\sum p\vec{J}^{h}h^{2}\right)\\
=-\dD(H\vec{U}^{*})+\vec{v}^{*}\cdot\dG H+\vec{\omega}^{*}\cdot\vec{J}^{h}\label{eq-numerical-scheme}
\end{multline}
where $\vec{J}^{h}=(\vec{x}-\vec{c})\times\dG H$. More precisely,
the terms on the left-hand side of Equation~\eqref{eq-numerical-scheme}
read as 
\begin{align*}
\left(\dD\left(\frac{H}{\rho}\dG p\right)\right)_{ij} & =\dfrac{1}{h}\left(\dfrac{H_{i+\frac{1}{2},j}}{\rho_{i+\frac{1}{2},j}}\,\dfrac{p_{i+1,j}-p_{ij}}{h}-\dfrac{H_{i-\frac{1}{2},j}}{\rho_{i-\frac{1}{2},j}}\,\dfrac{p_{ij}-p_{i-1,j}}{h}\right)\\
& \qquad+\dfrac{1}{h}\left(\dfrac{H_{i,j+\frac{1}{2}}}{\rho_{i,j+\frac{1}{2}}}\,\dfrac{p_{i,j+1}-p_{ij}}{h}-\dfrac{H_{i,j-\frac{1}{2}}}{\rho_{ij-\frac{1}{2}}}\,\dfrac{p_{ij}-p_{i,j-1}}{h}\right),\\
(\dG H)_{ij} & =\left(\dfrac{H_{i+\frac{1}{2},j}-H_{i-\frac{1}{2},j}}{h},\;\dfrac{H_{i,j+\frac{1}{2}}-H_{i,j-\frac{1}{2}}}{h}\right),\\
(\vec{J}^{h})_{ij} & =(\vec{\tilde{x}}_{ij}-\vec{c})\times(\dG H)_{ij}.
\end{align*}
where $\vec{\tilde{x}}_{ij}$ is given as 
\[
\vec{\tilde{x}}_{ij}=\begin{cases}
\frac{1}{2}(\vec{x}+\vec{y}) & \quad\mbox{if}\quad\partial C_{ij}\cap\partial\Omega=\set{\vec{x},\vec{y}},\\
0 & \quad\mbox{if}\quad\partial C_{ij}\cap\partial\Omega=\emptyset.
\end{cases}
\]

\subsection{Stability}

Let $L^{h}:\bR^{\abs{\Omega^{h}}}\to\bR^{\abs{\Omega^{h}}}$ be a
linear operator associated with the left-hand side of \eqref{eq-numerical-scheme}.
Then the linear system \eqref{eq-numerical-scheme} reads as: Given
a triple $(\vec{U}^{*},\vec{v}^{*},\vec{\omega}^{*})$, find a scalar
function $p:\Omega^{h}\to\bR$ satisfying 
\begin{equation}
L^{h}p=-\dD(H\vec{U}^{*})+\vec{v}^{*}\cdot\dG H+\vec{\omega}^{*}\cdot\vec{J}^{h}.\label{eq-p-equation-Lp}
\end{equation}
In order to verify the existence of $p$, we first estimate some properties
of $L^{h}.$

\begin{lemma} \label{lemma-SpSemiD-Lh} The linear operator $L^{h}$
	is symmetric and positive semi-definite on $\bR^{\abs{\Omega^{h}}}$
	and $\ker(L^{h})=\Span\set{1_{\Omega^{h}}}$ where $1_{\Omega^{h}}$
	is the function for which $1_{\Omega^{h}}\equiv1$ on $\Omega^{h}$.
\end{lemma} \begin{proof} Let $p^{(1)},p^{(2)}\in\bR^{\abs{\Omega^{h}}}$
be arbitrarily given. The discrete integration by parts shown in Lemma~\ref{lemma-DInt-by-Part}
implies 
\begin{align*}
& \ip{L^{h}p^{(1)},p^{(2)}}\\
& =-\sum_{i,j}p_{ij}^{(2)}\left(\dD\left(\frac{H}{\rho}\dG p^{(1)}\right)\right)_{ij}+\sum_{i,j}p_{ij}^{(2)}\left(\dG H\right)_{ij}\cdot\frac{1}{m}\left(\sum_{k,l}p_{kl}^{(1)}(\dG H)_{kl}\, h^{2}\right)\\
& \qquad\qquad+\sum_{i,j}p_{ij}^{(2)}(\vec{J}^{h})_{ij}\cdot\bI^{-1}\left(\sum_{k,l}p_{kl}^{(1)}(\vec{J}^{h})_{kl}\, h^{2}\right)\\
& =\sum_{i,j}\frac{H_{i+\frac{1}{2},j}}{\rho_{i+\frac{1}{2},j}}(\dG_{x}p^{(2)})_{i+\frac{1}{2},j}(\dG_{x}p^{(1)})_{i+\frac{1}{2},j}+\sum_{i,j}\frac{H_{i,j+\frac{1}{2}}}{\rho_{i,j+\frac{1}{2}}}(\dG_{y}p^{(2)})_{i,j+\frac{1}{2}}(\dG_{y}p^{(1)})_{i,j+\frac{1}{2}}\\
& \qquad\qquad+\frac{h^{2}}{m}\left(\sum_{i,j}p_{ij}^{(2)}(\dG H)_{ij}\right)\cdot\left(\sum_{k,l}p_{kl}^{(1)}(\dG H)_{kl}\right)\\
& \qquad\qquad+h^{2}\left(\sum_{i,j}p_{ij}^{(2)}(\vec{J}^{h})_{ij}\right)\cdot\bI^{-1}\left(\sum_{k,l}p_{kl}^{(1)}(\vec{J}^{h})_{kl}\right)\\
& \quad=\ip{p^{(1)},L^{h}p^{(2)}}
\end{align*}
and the symmetry of $\bI^{-1}$ shows that $L^{h}$ is symmetric.
In particular, if $p^{(1)}=p^{(2)}$, we have 
\begin{multline*}
\ip{L^{h}p^{(1)},p^{(1)}}=\sum_{i,j}\frac{H_{i+\frac{1}{2},j}}{\rho_{i+\frac{1}{2},j}}(\dG_{x}p^{(1)})_{i+\frac{1}{2},j}^{2}+\sum_{i,j}\frac{H_{i,j+\frac{1}{2}}}{\rho_{i,j+\frac{1}{2}}}(\dG_{y}p^{(1)})_{i,j+\frac{1}{2}}^{2}\\
+\frac{h^{2}}{m}\Abs{\sum_{i,j}p_{ij}^{(1)}(\dG H)_{ij}}^{2}+h^{2}\left(\sum_{i,j}p_{ij}^{(1)}(\vec{J}^{h})_{ij}\right)\cdot\bI^{-1}\left(\sum_{k,l}p_{kl}^{(1)}(\vec{J}^{h})_{k\ell}\right),
\end{multline*}
which implies that $\ip{L^{h}p,p}\ge0$ for all $p\in\bR^{\abs{\Omega^{h}}}$
because the inertia matrix $\bI$ is positive-definite; hence, $L^{h}$
is positive semi-definite. Let $p\in\ker(L^{h})$. Then we have $\ip{L^{h}p,p}=0$,
which implies 
\[
(\dG_{x}p)_{i+\frac{1}{2},j}=(\dG_{x}p)_{i+\frac{1}{2},j}=0,\quad\forall(x_{i},y_{j})\in\Omega^{h}.
\]
This implies that $p_{ij}=p_{i+1,j}=p_{i,j+1}$ for all $(x_{i},y_{j})\in\Omega^{h}$,
so that $p$ is constant. Conversely, if $p$ is a constant vector,
then Lemma~\ref{lemma-discrte-nj-zero} shows that $L^{h}p=0$. This
completes the proof. \end{proof}

The following theorem shows the existence and uniqueness condition
of the solution for the linear system $L^{h}p=f.$

\begin{theorem}\label{thm-solvability} The linear equation $L^{h}p=f$
	is solvable if and only if $\sum_{i,j}f_{ij}=0$. Furthermore, there
	exists a unique solution $p\in\set{1_{\Omega^{h}}}^{\perp}$. \end{theorem}
\begin{proof} Since $L^{h}$ is symmetric, Lemma~\ref{lemma-SpSemiD-Lh}
	implies that the range of $L^{h}$, denoted by $R(L^{h})$, is the
	orthogonal complement of $\Span\set{1_{\Omega^{h}}}$; that is, $R(L^{h})=\set{1_{\Omega^{h}}}^{\perp}$.
	Also, the lemma yields that $L^{h}$ is symmetric positive definite
	on $R(L^{h})$ so that for $f\in\{1_{\Omega^{h}}\}^{\perp},$ the
	equation $L^{h}p=f$ has a unique solution $p\in\set{1_{\Omega^{h}}}^{\perp}.$
\end{proof}

Theorem \ref{thm-solvability} may be regarded as an analogy of the
compatibility condition shown in Theorem \ref{thm-Pexistence} in
the following sense. Given a triple $(\vec{U}^{*},\vec{v}^{*},\vec{\omega}^{*})$,
consider the problem of finding a scalar function $p:\Omega^{h}\to\bR$
satisfying 
\begin{equation}
L^{h}p=-\dD\left(H\vec{U}^{*}\right)+\vec{v}^{*}\cdot\dG H+\vec{\omega}^{*}\cdot\vec{J}^{h}\quad\text{in}\quad\Omega^{h}.\label{eq-main-problem-3-7}
\end{equation}
Theorem \ref{thm-solvability} shows that the linear system \eqref{eq-main-problem-3-7}
is solvable if and only if 
\[
-\dD(H\vec{U}^{*})+\vec{v}^{*}\cdot\dG H+\vec{\omega}^{*}\cdot\vec{J}^{h}\in\set{1_{\Omega^{h}}}^{\perp},
\]
that is, 
\begin{equation}
\sum_{(x_{i},y_{j})\in\Omega^{h}}\left(-\dD(H\vec{U}^{*})+\vec{v}^{*}\cdot\dG H+\vec{\omega}^{*}\cdot\vec{J}^{h}\right)_{ij}=0,\label{eq3-7-discrete-solvability}
\end{equation}
Precisely, the problem \eqref{eq-main-problem-3-7} reads as 
\[
-\dD\left(\frac{1}{\rho}\dG p\right)=-\dD\vec{U}^{*}\quad\text{in}\quad\Omega_{\circ}^{h}
\]
and 
\begin{multline*}
-\dD\left(\frac{H}{\rho}\dG p\right)+\dG H\cdot\frac{1}{m}\left(\sum p\dG H\, h^{2}\right)+\vec{J}^{h}\cdot\mathbb{I}^{-1}\left(\sum p\vec{J}^{h}\, h^{2}\right)\\
=-\dD(H\vec{U}^{*})+\vec{v}^{*}\cdot\dG H+\vec{\omega}^{*}\cdot\vec{J}^{h}\quad\text{in}\quad\Omega_{\Gamma}^{h}.
\end{multline*}
Also, splitting the summation on the left-hand side of \eqref{eq3-7-discrete-solvability},
we have 
\[
-\sum_{(x_{i},y_{j})\in\Omega_{o}^{h}}(\dD\vec{U}^{*})_{ij}+\sum_{(x_{i},y_{j})\in\Omega_{\Gamma}^{h}}\left(-\dD(H\vec{U}^{*})+\vec{v}^{*}\cdot\dG H+\vec{\omega}^{*}\cdot\vec{J}^{h}\right)_{ij}=0.
\]
This is a discrete version of compatibility condition.

Once $p$ is solved, the triple $(\vec{U}^{*},\vec{v}^{*},\vec{\omega}^{*})$
is decomposed as 
\[
(\vec{U}^{*},\vec{v}^{*},\vec{\omega}^{*})=(\vec{U},\vec{v},\vec{\omega})+\left(\frac{1}{\rho}\dG p,\;\frac{1}{m}\sum_{k,l}(p\dG H)_{kl}\, h^{2},\;\bI^{-1}\sum_{k,l}(p\vec{J}^{h})_{kl}\, h^{2}\right).
\]

Now, we are to show that the decomposition is unique with $p$ satisfying
\eqref{eq-main-problem-3-7} and orthogonal with respect to the inner
product $\ip{\cdot,\cdot}_{E_{h}}$ defined by 
\begin{equation}
\Ip{(\vec{U}_{1}^{h},\vec{v}_{1}^{h},\vec{\omega}_{1}^{h}),(\vec{U}_{2}^{h},\vec{v}_{2}^{h},\vec{\omega}_{2}^{h})}_{E_{h}}
:=\int\frac{1}{2}\rho H\vec{U}_{1}^{h}\cdot\vec{U}_{2}^{h}\, d\Omega^{h}+\frac{1}{2}m\vec{v}_{1}^{h}\cdot\vec{v}_{2}^{h}+\frac{1}{2}\vec{\omega}_{1}^{h}\cdot\bI\vec{\omega}_{2}^{h}.\label{eq-Eh-innerproduct}
\end{equation}

\begin{theorem} \label{corollary-main} Given a triple $(\vec{U}^{*},\vec{v}^{*},\vec{\omega}^{*})$,
	there exists a unique $p\in\set{1_{\Omega^{h}}}^{\perp}$ satisfying
	\begin{equation}
	L^{h}p=-\dD(H\vec{U}^{*})+\vec{v}^{*}\cdot\dG H+\vec{\omega}^{*}\cdot\vec{J}^{h}\quad\text{in}\quad\Omega^{h}.\label{eq-Lp-equation}
	\end{equation}
	Therefore, the triple $(\vec{U}^{*},\vec{v}^{*},\vec{\omega}^{*})$
	is uniquely decomposed as 
	\begin{equation}
	(\vec{U}^{*},\vec{v}^{*},\vec{\omega}^{*})=(\vec{U},\vec{v},\vec{\omega})+\left(\frac{1}{\rho}\dG p,\;\frac{1}{m}\sum_{k,l}(p\dG H)_{kl}\, h^{2},\;\bI^{-1}\sum_{k,l}(p\vec{J}^{h})_{kl}h^{2}\right)\label{eq-main-discrete-decomposition2}
	\end{equation}
	with $p$ satisfying \eqref{eq-Lp-equation}. Furthermore, the decomposition
	is orthogonal with respect to the inner product \eqref{eq-Eh-innerproduct}.
\end{theorem} \begin{proof} Lemmas \ref{lemma-DInt-by-Part} and
\ref{lemma-discrte-nj-zero} show the solvability condition 
\begin{multline*}
\sum_{(x_{i},y_{j})\in\Omega^{h}}\left(-\dD(H\vec{U}^{*})+\vec{v}^{*}\cdot\dG H+\vec{\omega}^{*}\cdot\vec{J}^{h}\right)_{ij}\\
=\int(H\vec{U}^{*}\cdot\dG1_{\Omega^{h}})\, d\Omega^{h}+\vec{v}^{*}\cdot\int(\dG H)\, d\Omega^{h}+\vec{\omega}^{*}\cdot\int\vec{J}^{h}\, d\Omega^{h}=0.
\end{multline*}
Then, Theorem \ref{thm-solvability} verifies the existence and uniqueness
of $p$ in $\set{1_{\Omega^{h}}}^{\perp}$. With such $p$, we decompose
$(\vec{U}^{*},\vec{v}^{*},\vec{w}^{*})$ as 
\[
(\vec{U}^{*},\vec{v}^{*},\vec{\omega}^{*})=(\vec{U},\vec{v},\vec{\omega})+\left(\frac{1}{\rho}\dG p,\;\frac{1}{m}\sum_{k,l}(p\dG H)_{kl}\, h^{2},\;\bI^{-1}\sum_{k,l}(p\vec{J}^{h})_{kl}\, h^{2}\right).
\]
Applying the decomposition to \eqref{eq-Lp-equation} shows 
\[
-\dD(H\vec{U})+\dG H\cdot\vec{v}+\vec{J}^{h}\cdot\vec{\omega}=0\quad\text{in}\quad\Omega^{h}.
\]
Using the identity above and Lemma \ref{lemma-DInt-by-Part}, we have
\begin{align*}
& \Ip{(\vec{U},\vec{v},\vec{\omega}),\left(\frac{1}{\rho}\dG p,\;\frac{1}{m}\sum_{k,l}(p\dG H)_{kl}\, h^{2},\;\bI^{-1}\sum_{k,l}(p\vec{J}^{h})_{kl}\, h^{2}\right)}_{E_{h}}\\
& \quad=\int\frac{1}{2}\rho H\vec{U}\cdot\left(\frac{1}{\rho}\dG p\right)\, d\Omega^{h}+\frac{1}{2}m\vec{v}\cdot\left(\frac{1}{m}\int p\dG H\, d\Omega^{h}\right)+\frac{1}{2}\vec{\omega}\cdot\left(\int p\vec{J}^{h}\, d\Omega^{h}\right)\\
& \quad=\frac{1}{2}\int p(-\dD(H\vec{U})+\vec{v}\cdot\dG H+\vec{\omega}\cdot\vec{J}^{h})\, d\Omega^{h}=0.
\end{align*}
This shows the orthogonality of the decomposition with respect to
the inner product \eqref{eq-Eh-innerproduct}. The uniqueness of the
decomposition is verified from Theorem \ref{thm-solvability}, which
completes the proof. \end{proof}

The following theorem demonstrates that the discrete projection $(\vec{U},\vec{v},\vec{\omega})$
of $(\vec{U}^{*},\vec{v}^{*},\vec{\omega}^{*})$ is stable in the
sense that it does not increase the kinetic energy:

\begin{theorem} \label{thm-discrete-energy} Assume the kinetic energy
	of a triple $(\vec{U}^{*},\vec{v}^{*},\vec{\omega}^{*})$ with $\vec{U}^{*}=(U_{x}^{*},U_{y}^{*})$
	is discretized as 
	\begin{align*}
	E_{h} & =\frac{1}{2}\int(\rho H\vec{U}^{*}\cdot\vec{U}^{*})\, d\Omega^{h}+\frac{1}{2}m\vec{v}^{*}\cdot\vec{v}^{*}+\frac{1}{2}\vec{\omega}^{*}\cdot\bI\vec{\omega}^{*}\\
	& =\frac{1}{2}\sum_{i,j}\left((\rho H(U_{x}^{*})^{2})_{i+\frac{1}{2},j}+(\rho H(U_{y}^{*})^{2})_{i,j+\frac{1}{2}}\right)h^{2}+\frac{1}{2}m\abs{\vec{v}^{*}}^{2}+\frac{1}{2}\vec{\omega}^{*}\bI\vec{\omega}^{*}.
	\end{align*}
	If the system $(\vec{U}^{*},\vec{v}^{*},\vec{\omega}^{*})$ is projected
	into $(\vec{U},\vec{v},\vec{\omega})$ given by 
	\[
	(\vec{U},\vec{v},\vec{\omega})=(\vec{U}^{*},\vec{v}^{*},\vec{\omega}^{*})-\left(\frac{1}{\rho}\dG p,\;\frac{1}{m}\sum_{k,l}(p\dG H)_{kl}\, h^{2},\;\bI^{-1}\sum_{k,l}(p\vec{J}^{h})_{kl}\, h^{2}\right)
	\]
	where $p\in\{1_{\Omega^{h}}\}^{\perp}$ is the solution to the problem
	\eqref{eq-Lp-equation}. 
	Then, the discrete projection is stable in the sense $E_{h}(\vec{U}^{*},\vec{v}^{*},\vec{\omega}^{*})\ge E_{h}(\vec{U},\vec{v},\vec{\omega})$.
\end{theorem} \begin{proof} From the orthogonality shown in Theorem
\ref{corollary-main}, we conclude 
\begin{equation*}
E_{h}(\vec{U}^{*},\vec{v}^{*},\vec{\omega}^{*})=\Ip{(\vec{U}^{*},\vec{v}^{*},\vec{\omega}^{*}),(\vec{U}^{*},\vec{v}^{*},\vec{\omega}^{*})}_{E_{h}}
\ge\Ip{(\vec{U},\vec{v},\vec{\omega}),(\vec{U},\vec{v},\vec{\omega})}_{E_{h}}=E_{h}(\vec{U},\vec{v},\vec{\omega})
\end{equation*}
and this shows the decrease of the kinetic energy. \end{proof}

\section{Convergence analysis}\label{sec5}

In this section, we estimate the consistency and convergence of the
numerical scheme. In the previous sections, we have decomposed a given
triple $(\vec{U}^{*},\vec{v}^{*},\vec{\omega}^{*})$ using the Heaviside
function $H=\chi_{\Omega}$ as 
\begin{equation}
(\vec{U}^{*},\vec{v}^{*},\vec{\omega}^{*})=(\vec{U},\vec{v},\vec{\omega})+\left(\frac{1}{\rho}\nabla p,\;-\frac{1}{m}\int_{\Gamma}p\vec{n}\, dS,\;-\mathbb{I}^{-1}\int_{\Gamma}p\vec{J}\, dS\right)\label{eq-continuous-decomposition}
\end{equation}
by solving the equation 
\begin{equation}
\mathcal{L}(\vec{U},\vec{v},\vec{\omega}):=-\nabla\cdot(H\vec{U})+\nabla H\cdot\vec{v}+(\vec{x}-\vec{c})\times\nabla H\cdot\vec{\omega}=0\quad\text{in}\quad\mathbb{R}^{d}\label{eq-continuous-Pb}
\end{equation}
with the conditions 
\[
\left\{ \begin{aligned}\nabla\cdot\vec{U} & =0\quad\text{in}\quad\Omega\\
\vec{U}\cdot\vec{n} & =\vec{v}\cdot\vec{n}+\vec{\omega}\cdot\vec{J}\quad\text{on}\quad\Gamma=\partial\Omega.
\end{aligned}
\right.
\]
Then the numerical approximation $(\vec{U}^{h},\vec{v}^{h},\vec{\omega}^{h})$
to $(\vec{U},\vec{v},\vec{\omega})$ has been obtained by solving
the linear system 
\begin{equation}
\mathcal{L}^{h}(\vec{U}^{h},\vec{v}^{h},\vec{\omega}^{h}):=-\dD(H\vec{U}^{h})+\dG H\cdot\vec{v}^{h}+\vec{J}^{h}\cdot\vec{\omega}^{h}=0\quad\text{in}\quad\Omega^{h}\label{eq-discrete-Pb}
\end{equation}
using the decomposition 
\begin{equation}
(\vec{U}^{*},\vec{v}^{*},\vec{\omega}^{*})=(\vec{U}^{h},\vec{v}^{h},\vec{\omega}^{h})
+\left(\frac{1}{\rho}\dG p^{h},\;\frac{1}{m}\sum_{i,j,k}(p^{h}\dG H)_{ijk}h^{3},\;\mathbb{I}^{-1}\sum_{i,j,k}(p^{h}\vec{J}^{h})_{ijk}h^{3}\right)\label{eq-discrete-decomposition}
\end{equation}
where $p^{h}\in\set{1_{\Omega^{h}}}^{\perp}$.

By $(\vec{U},\vec{v},\vec{\omega})$ and $(\vec{U}^{h},\vec{v}^{h},\vec{\omega}^{h})$,
throughout this section, we denote the continuous and numerical solutions,
respectively. In the setting, let $(\vec{U}_{e},\vec{v}_{e},\vec{\omega}_{e}):=(\vec{U}-\vec{U}^{h},\vec{v}-\vec{v}^{h},\vec{\omega}-\vec{\omega}^{h})$
denote the convergence error. The consistency error for numerical
scheme is defined as 
\[
c^{h}:=\mathcal{L}^{h}(\vec{U},\vec{v},\vec{\omega})-\mathcal{L}^{h}(\vec{U}^{h},\vec{v}^{h},\vec{\omega}^{h}).
\]
In order to estimate the consistency, we need the following lemma.

\begin{lemma}\label{lemma-n-GH-J-Jh} We have the followings. 
	\begin{itemize}
		\item[(i)] For $d=2,$ we have on $C_{ij}\cap\partial\Omega$ 
		\[
		\intop_{C_{ij}\cap\Gamma}\vec{n}\, dS=-(\dG H)_{ij}h^{2}\quad\text{and}\quad\intop_{C_{ij}\cap\Gamma}\vec{J}\, dS=-\left(\vec{J}^{h}\right)_{ij}h^{2}
		\]
		where $(\vec{J}^{h})_{ij}=(0,0)$ if $\partial C_{ij}\cap\Gamma=\emptyset$
		and $(\vec{J}^{h})_{ij}=(\frac{\vec{x}_{1}+\vec{x}_{2}}{2}-\vec{c})\times\left(\dG H\right)_{ij}$
		if $\partial C_{ij}\cap\Gamma=\set{\vec{x}_{1},\vec{x}_{2}}$. 
		\item[(ii)] For $d=3,$ we have on $C_{ijk}\cap\partial\Omega$ 
		\[
		\intop_{C_{ijk}\cap\Gamma}\vec{n}\, dS=-(\dG H)_{ijk}h^{3}\quad\text{and}\quad\intop_{C_{ijk}\cap\Gamma}\vec{J}\, dS=-(\vec{J}^{h})_{ijk}h^{3}+h^{3}\vec{\varepsilon}_{ijk}
		\]
		where $(\vec{J}^{h})_{ijk}=(0,0,0)$ if $\partial C_{ijk}\cap\Gamma=\emptyset$
		and $(\vec{J}^{h})_{ijk}=(\vec{x}_{ijk}-\vec{c})\times(\dG H)_{ijk}$
		if $\partial C_{ijk}\cap\Gamma\neq\emptyset,$ and 
		\begin{equation}
		h^{3}\vec{\varepsilon}_{ijk}=\intop_{C_{ijk}\cap\Gamma}(\vec{x}-\vec{x}_{ijk})\times\vec{n}\, dS.\label{eq-J-error}
		\end{equation}
		
		\item[(iii)] For $d=2,3$ we have on $\Gamma=\partial\Omega$ 
		\begin{equation}
		\Abs{\int_{\Gamma}p\vec{n}\, dS+\sum_{\vec{x}_{h}\in\Omega_{\Gamma}^{h}}(pGH)_{\vec{x}_{h}}h^{d}}\le\frac{\sqrt{d}\,\abs{\Gamma}\,\norm{\nabla p}_{L^{\infty}}}{2}\, h,\label{eq-n-nh-error}
		\end{equation}
		\begin{equation}
		\Abs{\int_{\Gamma}p\vec{J}\, dS+\sum_{\vec{x}_{h}\in\Omega_{\Gamma}^{h}}(p\vec{J}^{h})_{\vec{x}_{h}}h^{d}}\le\frac{\sqrt{d}\,\abs{\Gamma}\,\bigl((\diam\Gamma)\,\norm{\nabla p}_{L^{\infty}}+\delta_{d,3}\norm{p}_{L^{\infty}}\bigr)}{2}\, h,\label{eq-J-Jh-error}
		\end{equation}
		where $\delta$ is the Kronecker delta symbol. 
	\end{itemize}
\end{lemma}

\begin{proof} 
	\begin{itemize}
		\item[(i)] Let $\vec{n}$ be the outward unit normal vector of $\partial\left(C_{ij}\cap\Omega\right)$.
		Then the same argument used for the proof of Lemma \ref{lemma-zero-nj}
		shows 
		\[
		\intop_{\partial\left(C_{ij}\cap\Omega\right)}\vec{n}\, dS=0\quad\text{and}\quad\intop_{\partial\left(C_{ij}\cap\Omega\right)}\vec{J}\, dS=0.
		\]
		These imply 
		\[
		\intop_{C_{ij}\cap\Gamma}\vec{n}\, dS=-\intop_{\partial C_{ij}\cap\Omega}\vec{n}\, dS=-(\dG H)_{ij}\, h^{2}
		\]
		and 
		\[
		\intop_{C_{ij}\cap\Gamma}\vec{J}\, dS=-\intop_{\partial C_{ij}\cap\Omega}\vec{J}\, dS=-(\vec{J}^{h})_{ij}\, h^{2},
		\]
		where $(\vec{J}^{h})_{ij}=(0,0)$ if $\partial C_{ij}\cap\Gamma=\emptyset$
		and $(\vec{J}^{h})_{ij}=(\frac{\vec{x}_{1}+\vec{x}_{2}}{2}-\vec{c})\times(\dG H)_{ij}$
		if $\partial C_{ij}\cap\Gamma=\set{\vec{x}_{1},\vec{x}_{2}}$.
		\item[(ii)] Using the similar argument used for the proof of (i), one can show
		\[
		\intop_{C_{ijk}\cap\Gamma}\vec{n}\, dS=-(\dG H)_{ijk}\, h^{3}\quad\text{and}\quad\intop_{C_{ijk}\cap\Gamma}\vec{J}\, dS=-(\vec{J}^{h})_{ijk}\, h^{3}+h^{3}\vec{\varepsilon}_{ijk}
		\]
		with $\vec{\varepsilon}_{ijk}$ given in \eqref{eq-J-error}.
		\item[(iii)] Using {(i)} and {(ii)}, we have 
		\begin{align*}
		\Abs{\int_{\Gamma}p\vec{n}\, dS+\sum_{\vec{x}_{h}\in\Omega_{\Gamma}^{h}}(pGH)_{\vec{x}_{h}}h^{d}} & \le\sum_{\vec{x}_{h}\in\Omega_{\Gamma}^{h}}\;\;\int_{\Gamma\cap C_{\vec{x}_{h}}}\abs{p(\vec{x})-p(\vec{x}_{h})}\, dS\\
		& \le\sum_{\vec{x}_{h}\in\Omega_{\Gamma}^{h}}\frac{\sqrt{d}\,\norm{\nabla p}_{L^{\infty}}}{2}\, h\int_{\Gamma\cap C_{\vec{x}_{h}}}\, dS\\
		& =\frac{\sqrt{d}\,\abs{\Gamma}\,\norm{\nabla p}_{L^{\infty}}}{2}\, h.
		\end{align*}
		Here, we used that for $\vec{x}_{h}\in\Omega_{\Gamma}^{h}$ and $\vec{x}\in\Gamma\cap C_{\vec{x}_{h}}$,
		we have 
		\[
		\abs{p(\vec{x})-p(\vec{x}_{h})}\le\abs{\vec{x}-\vec{x}_{h}}\,\norm{\nabla p}_{L^{\infty}}\le\frac{\sqrt{d}\,\norm{\nabla p}_{L^{\infty}}}{2}\, h.
		\]
		So, we have shown the inequality \eqref{eq-n-nh-error}. To show \eqref{eq-J-Jh-error},
		we apply (i) and (ii) to obtain 
		\begin{align*}
		\int_{\Gamma}p\vec{J}\, dS+\sum_{\vec{x}_{h}\in\Omega_{\Gamma}^{h}}\left(p\vec{J}^{h}\right)_{\vec{x}_{h}}h^{d} & =\sum_{\vec{x}_{h}\in\Omega_{\Gamma}^{h}}\;\;\int_{\Gamma\cap C_{\vec{x}_{h}}}(p(\vec{x})-p(\vec{x}_{h}))\vec{J}dS\\
		& \qquad+\delta_{d,3}\sum_{\vec{x}_{h}\in\Omega_{\Gamma}^{h}}\;\;\int_{\Gamma\cap C_{\vec{x}_{h}}}p(\vec{x}_{h})(\vec{x}-\vec{x}_{h})\times\vec{n}\, dS.
		\end{align*}
		Applying the similar argument to obtain \eqref{eq-n-nh-error} leads
		to the inequality \eqref{eq-J-Jh-error}. 
	\end{itemize}
\end{proof}

\begin{theorem}[Consistency error]\label{them-consistency} Let $(\vec{U},\vec{v},\vec{\omega})$
	be the continuous solution to \eqref{eq-continuous-decomposition}
	- \eqref{eq-continuous-Pb} and $(\vec{U}^{h},\vec{v}^{h},\vec{\omega}^{h})$
	be the numerical solution to \eqref{eq-discrete-Pb} - \eqref{eq-discrete-decomposition}.
	Then we have the consistency error $c^{h}=\mathcal{L}^{h}(\vec{U},\vec{v},\vec{\omega})-\mathcal{L}^{h}(\vec{U}^{h},\vec{v}^{h},\vec{\omega}^{h})$
	as 
	\begin{equation}
	(c^{h})_{\vec{x}_{h}}=\begin{cases}
	O(h^{2}), & \quad\text{if}\quad\vec{x}_{h}\in\Omega_{\circ}^{h}\\
	O(1), & \quad\text{if}\quad\vec{x}_{h}\in\Omega_{\Gamma}^{h}.
	\end{cases}\label{eq-consistency-estimate}
	\end{equation}
\end{theorem}

\begin{proof} We give the proof for $d=2;$ the case for $d=3$ is
	shown in the similar argument. For each cell $C_{ij}$, the divergence
	theorem gives 
	\begin{equation}
	0=\intop_{C_{ij}\cap\Omega}\nabla\cdot\vec{U}\, dx=\intop_{\partial(C_{ij}\cap\Omega)}\vec{U}\cdot\vec{n}\, dS
	=\intop_{\partial C_{ij}\cap\Omega}\vec{U}\cdot\vec{n}\, dS+\intop_{C_{ij}\cap\Gamma}\left(\vec{v}\cdot\vec{n}+\vec{\omega}\cdot\vec{J}\right)\, dS.\label{eq-nj-expresion}
	\end{equation}
	Let $\vec{U}=(u,v)$. On $\partial C_{ij}\cap\Omega,$ we have 
	\begin{multline*}
	\intop_{\partial C_{ij}\cap\Omega}\vec{U}\cdot\vec{n}\, dS=\intop_{E_{i+\frac{1}{2},j}\cap\Omega}u(x_{i+\frac{1}{2}},y)\, dy\;-\intop_{E_{i-\frac{1}{2},j}\cap\Omega}u(x_{i-\frac{1}{2}},y)\, dy\\
	+\intop_{E_{i,j+\frac{1}{2}}\cap\Omega}v(x,y_{i+\frac{1}{2}})\, dx\;-\intop_{E_{i,j-\frac{1}{2}}\cap\Omega}v(x,y_{i-\frac{1}{2}})\, dx.
	\end{multline*}
	Applying the Taylor series expansion at the middle point of the edge,
	we estimate the integrals on the right-hand side 
	\begin{align*}
	\intop_{E_{i+\frac{1}{2},j}\cap\Omega}u(x_{i+\frac{1}{2}},y)\, dy & -hH_{i+\frac{1}{2},j}u(x_{i+\frac{1}{2}},y_{j})\\[-12pt]
	& =\intop_{E_{i+\frac{1}{2},j}\cap\Omega}\left(u(x_{i+\frac{1}{2}},y)-u(x_{i+\frac{1}{2}},y_{j})\right)\, dy\\
	& =u_{y}(x_{i+\frac{1}{2}},y_{j})\,\intop_{E_{i+\frac{1}{2},j}\cap\Omega}\left(y-y_{j}\right)\, dy+O(h^{3}).
	\end{align*}
	In particular, if $H_{i\pm\frac{1}{2},j}=1,$ then we have 
	\[
	\intop_{E_{i+\frac{1}{2},j}\cap\Omega}u(x_{i+\frac{1}{2}},y)\, dy\;-\intop_{E_{i-\frac{1}{2},j}\cap\Omega}u(x_{i+\frac{1}{2}},y)\, dy
	=h\, u(x_{i+\frac{1}{2}},y_{j})-h\, u(x_{i-\frac{1}{2}},y_{j})+O(h^{4}).
	\]
	We obtain the similar estimations for the other edges. Combining the
	estimations, we have 
	\begin{equation}
	\intop_{\partial C_{ij}\cap\Omega}\vec{U}\cdot\vec{n}\, dS-(\dD(H\vec{U}))_{ij}\, h^{2}=\begin{cases}
	O(h^{4}) & \mbox{if }\;(x_{i},y_{j})\in\Omega_{\circ}^{h}\\
	O(h^{2}) & \mbox{if }\;(x_{i},y_{j})\in\Omega_{\Gamma}^{h}.
	\end{cases}\label{eq-edge-estimate-2D}
	\end{equation}
	Since $\mathcal{L}^{h}(\vec{U}^{h},\vec{v}^{h},\vec{\omega}^{h})=0,$
	using Lemma \ref{lemma-n-GH-J-Jh} and Equation \eqref{eq-nj-expresion},
	we have the consistency error $(c^{h})_{ij}$ at $(x_{i},y_{j})\in\Omega^{h}$
	as 
	\begin{align*}
	h^{2}(c^{h})_{ij} & =h^{2}\left(\mathcal{L}^{h}(\vec{U},\vec{v},\vec{\omega})-\mathcal{L}^{h}(\vec{U}^{h},\vec{v}^{h},\vec{\omega}^{h})\right)_{ij}=h^{2}\left(\mathcal{L}^{h}(\vec{U},\vec{v},\vec{\omega})\right)_{ij}\\
	& =-h^{2}(\dD(H\vec{U}))_{ij}+h^{2}(\dG H)_{ij}\cdot\vec{v}+h^{2}(\vec{J}^{h})_{ij}\cdot\vec{\omega}\\
	& =-h^{2}(\dD(H\vec{U}))_{ij}-\int_{C_{ij}\cap\Gamma}(\vec{v}\cdot\vec{n}+\vec{\omega}\cdot\vec{J})\, dS\\[-6pt]
	& =\int_{\partial C_{ij}\cap\Omega}\vec{U}\cdot\vec{n}\, dS-(\dD(H\vec{U}))_{ij}h^{2}.
	\end{align*}
	Then the estimation \eqref{eq-edge-estimate-2D} shows the consistency
	error \eqref{eq-consistency-estimate} for $d=2$.
	
	Now, we consider the case for $d=3$. A similar argument used for
	$d=2$ shows that for each cell $C_{ijk}$, we have 
	\[
	\intop_{\partial C_{ijk}\cap\Omega}\vec{U}\cdot\vec{n}\, dS=-\intop_{C_{ijk}\cap\Gamma}\left(\vec{v}\cdot\vec{n}+\vec{\omega}\cdot\vec{J}\right)\, dS.
	\]
	and 
	\begin{equation}
	\intop_{\partial C_{ijk}\cap\Omega}\vec{U}\cdot\vec{n}\, dS-(\dD(H\vec{U}))_{ijk}\, h^{3}=\begin{cases}
	O(h^{5}) & \mbox{if }\;(x_{i},y_{j},z_{k})\in\Omega_{\circ}^{h}\\
	O(h^{3}) & \mbox{if }\;(x_{i},y_{j},z_{k})\in\Omega_{\Gamma}^{h}.
	\end{cases}\label{eq-edge-estimate-3D}
	\end{equation}
	Since $L^{h}(\vec{U}^{h},\vec{v}^{h},\vec{\omega}^{h})=0$, we have
	the consistency error $(c^{h})_{ijk}$ at $(x_{i},y_{j},z_{k})\in\Omega^{h}$
	as 
	\begin{align*}
	h^{3}(c^{h})_{ijk} & =h^{3}\left(\mathcal{L}^{h}(\vec{U},\vec{v},\vec{\omega})-\mathcal{L}^{h}(\vec{U}^{h},\vec{v}^{h},\vec{\omega}^{h})\right)_{ijk}\\
	& =-h^{3}(\dD(H\vec{U}))_{ijk}+h^{3}(\dG H)_{ijk}\cdot\vec{v}+h^{3}(\vec{J}^{h})_{ijk}\cdot\vec{\omega}\\
	& =-h^{3}(\dD(H\vec{U}))_{ijk}-\intop_{C_{ijk}\cap\Gamma}\left(\vec{v}\cdot\vec{n}+\vec{\omega}\cdot\vec{J}\right)\, dS+h^{3}\vec{\varepsilon}_{ijk}\cdot\vec{\omega}\\[-6pt]
	& =\intop_{\partial C_{ijk}\cap\Omega}\vec{U}\cdot\vec{n}\, dS-(\dD(H\vec{U}))_{ijk}\, h^{3}+h^{3}\vec{\varepsilon}_{ijk}\cdot\vec{\omega}.
	\end{align*}
	From \eqref{eq-J-error}, we see that $\abs{\vec{\varepsilon}_{ijk}}=O(1)$
	if $(x_{i},y_{j},z_{k})\in\Omega_{\Gamma}^{h}$ and $\vec{\varepsilon}_{ijk}=0$
	if $(x_{i},y_{j},z_{k})\in\Omega_{o}^{h}$. From the estimation \eqref{eq-edge-estimate-3D},
	we shows the consistency error \eqref{eq-consistency-estimate} for
	$d=3$, which completes the proof. \end{proof}

The proof of Theorem \ref{them-consistency} reveals that there exists
a vector field 
\[
\vec{d}^{h}=\begin{cases}
(d_{i+\frac{1}{2},j}^{h},\; d_{i,j+\frac{1}{2}}^{h}) & \mbox{for}\; d=2\\
(d_{i+\frac{1}{2},j,k}^{h},\; d_{i,j+\frac{1}{2},k}^{h},\; d_{i,j,k+\frac{1}{2}}^{h}) & \mbox{for}\; d=3
\end{cases}
\]
such that for each $\vec{x}_{h}\in\Omega^{h}$ 
\[
h^{d}\left(\dD(H\vec{d}^{h})\right)_{\vec{x}_{h}}=h^{d}\left(\mathcal{L}^{h}(\vec{U},\vec{v},\vec{\omega})\right)_{\vec{x}_{h}}.
\]
For example, $d_{i+\frac{1}{2},j}^{h}$ is given as 
\[
hH_{i+\frac{1}{2},j}\; d_{i+\frac{1}{2},j}^{h}=\intop_{E_{i+\frac{1}{2},j}\cap\Omega}\left(\vec{U}(x_{i+\frac{1}{2}},y)-\vec{U}(x_{i+\frac{1}{2}},y_{j})\right)\cdot\vec{n}\, dy
\]
and $d_{i+\frac{1}{2},j,k}^{h}$ is given as 
\begin{multline*}
h^{2}H_{i+\frac{1}{2},j,k}\; d_{i+\frac{1}{2},j,k}^{h}=\intop_{F_{i+\frac{1}{2},j,k}\cap\Omega}\Big\{\bigl(\vec{U}(x_{i+\frac{1}{2}},y,z)-\vec{U}(x_{i+\frac{1}{2}},y_{j},z_{k})\bigr)\cdot\vec{n}\\[-6pt]
-\bigl(\omega_{2}(z-z_{k})-\omega_{3}(y-y_{j})\bigr)\Big\} dydz,
\end{multline*}
where $\vec{\omega}=\left(\omega_{1},\omega_{2},\omega_{3}\right)$
and we used the fact from Lemma \ref{lemma-zero-nj} that 
\[
h^{3}\vec{\varepsilon}_{ijk}\cdot\vec{\omega}=\vec{\omega}\cdot\intop_{C_{ijk}\cap\Gamma}(\vec{x}-\vec{x}_{ijk})\times\vec{n}\, dS=-\vec{\omega}\cdot\intop_{\partial C_{ijk}\cap\Omega}(\vec{x}-\vec{x}_{ijk})\times\vec{n}\, dS.
\]

It is not difficult to see that 
\begin{equation}
\begin{aligned}d_{i+\frac{1}{2},j}^{h} & =\begin{cases}
O(h) & \mbox{if}\quad0<H_{i+\frac{1}{2},j}<1\\
O(h^{2}) & \mbox{if}\quad H_{i+\frac{1}{2},j}=1.
\end{cases}\\
d_{i+\frac{1}{2},j,k}^{h} & =\begin{cases}
O(h) & \mbox{if}\quad0<H_{i+\frac{1}{2},j,k}<1\\
O(h^{2}) & \mbox{if}\quad H_{i+\frac{1}{2},j,k}=1.
\end{cases}
\end{aligned}
\label{eq-dh}
\end{equation}
and the same holds for $d_{i,j+\frac{1}{2}}$, $d_{i,j+\frac{1}{2},k}$,
and $d_{i,j,k+\frac{1}{2}}$.

Note that from the definition of $\vec{d}^{h},$ we have 
\[
\mathcal{L}^{h}(\vec{U}-\vec{U}^{h}-\vec{d}^{h},\vec{v}-\vec{v}^{h},\vec{\omega}-\vec{\omega}^{h})=0.
\]

\begin{theorem}[Convergence error] \label{them-convergence} Let
	$(\vec{U},\vec{v},\vec{\omega})$ be a continuous solution to \eqref{eq-continuous-decomposition}
	- \eqref{eq-continuous-Pb}, and $(\vec{U}^{h},\vec{v}^{h},\vec{\omega}^{h})$
	a numerical solution to \eqref{eq-discrete-Pb} - \eqref{eq-discrete-decomposition}.
	Then we have the convergence error\\ $(\vec{U}_{e},\vec{v}_{e},\vec{\omega}_{e}):=(\vec{U}-\vec{U}^{h},\vec{v}-\vec{v}^{h},\vec{\omega}-\vec{\omega}^{h})$
	with 
	\[
	\norm{(\vec{U}_{e},\vec{v}_{e},\vec{\omega}_{e})}_{E_{h}}=\left(\frac{1}{2}\int\rho H\vec{U}_{e}\cdot\vec{U}_{e}\, d\Omega^{h}+\frac{1}{2}m\vec{v}_{e}\cdot\vec{v}_{e}+\frac{1}{2}\vec{\omega}_{e}\cdot\mathbb{I}\vec{\omega}_{e}\right)^{\frac{1}{2}}=O(h).
	\]
\end{theorem}

\begin{proof} From the decompositions \eqref{eq-continuous-decomposition}
	and \eqref{eq-discrete-decomposition}, we have 
	\begin{multline*}
	\left(\begin{array}{c}
	\vec{U}_{e}-\vec{d}^{h}\\[9pt]
	\vec{v}_{e}\\[9pt]
	\vec{\omega}_{e}
	\end{array}\right)=\left(\begin{array}{c}
	{\displaystyle \frac{1}{\rho}\dG p^{h}-\frac{1}{\rho}\nabla p-\vec{d}^{h}}\\[9pt]
	{\displaystyle \frac{1}{m}\int(p^{h}\dG H)\, d\Omega^{h}+\frac{1}{m}\int_{\Gamma}p\vec{n}\, dS}\\[9pt]
	{\displaystyle \bI^{-1}\int(p^{h}\vec{J}^{h})\, d\Omega^{h}+\bI^{-1}\int_{\Gamma}p\vec{J}\, dS}
	\end{array}\right)\\[6pt]
	=\left(\begin{array}{c}
	{\displaystyle \frac{1}{\rho}(\dG p^{h}-\dG p)}\\[9pt]
	{\displaystyle \frac{1}{m}{\displaystyle \int(p^{h}-p)\dG H\, d\Omega^{h}}}\\[9pt]
	{\displaystyle \bI^{-1}\int(p^{h}-p)\vec{J}^{h}\, d\Omega^{h}}
	\end{array}\right)+\left(\begin{array}{c}
	{\displaystyle \frac{1}{\rho}\dG p-\frac{1}{\rho}\nabla p-\vec{d}^{h}}\\[9pt]
	{\displaystyle \frac{1}{m}\int_{\Gamma}p\vec{n}\, dS+\frac{1}{m}\int(p\dG H)\, d\Omega^{h}}\\[9pt]
	{\displaystyle \bI^{-1}\int_{\Gamma}p\vec{J}\, dS+\bI^{-1}\int p\vec{J}^{h}\, d\Omega^{h}}
	\end{array}\right)
	\end{multline*}
	We have shown that $\mathcal{L}^{h}(\vec{U}_{e}-\vec{d}^{h},\vec{v}_{e},\vec{\omega}_{e})=0$.
	Now we show that $(\vec{U}_{e}-\vec{d}^{h},\vec{v}_{e},\vec{\omega}_{e})$
	is the projection of $(\vec{\Lambda}_{\vec{U}}-\vec{d}^{h},\vec{\Lambda}_{\vec{v}},\vec{\Lambda}_{\vec{\omega}})$,
	where 
	\begin{align*}
	\vec{\Lambda}_{\vec{U}} & :=\frac{1}{\rho}\dG p-\frac{1}{\rho}\nabla p,\\
	\vec{\Lambda}_{\vec{v}} & :=\frac{1}{m}\int_{\Gamma}p\vec{n}\, dS+\frac{1}{m}\int p\dG H\, d\Omega^{h},\\
	\vec{\Lambda}_{\vec{\omega}} & :=\bI^{-1}\int_{\Gamma}p\vec{J}\, dS+\mathbb{I}^{-1}\int p\vec{J}^{h}\, d\Omega^{h},
	\end{align*}
	which follows if we prove that $(\vec{U}_{e}-\vec{d}^{h},\vec{v}_{e},\vec{\omega}_{e})$
	and 
	\[
	\left(\frac{1}{\rho}(\dG p^{h}-\dG p),\;\frac{1}{m}\int(p^{h}-p)\dG H\, d\Omega^{h},\;\bI^{-1}\int(p^{h}-p)\vec{J}^{h}\, d\Omega^{h}\right)
	\]
	are orthogonal with respect to the inner product $\ip{\cdot,\cdot}_{E_{h}}$.
	Indeed, the discrete integration by parts yields that their inner
	product equals 
	\begin{multline*}
	\frac{1}{2}\int(-\dD H(\vec{U}_{e}-\vec{d}^{h})+\vec{v}_{e}\cdot\dG H+\vec{\omega}_{e}\cdot\vec{J}^{h})(p^{h}-p)\, d\Omega^{h}\\
	=\frac{1}{2}\int\mathcal{L}^{h}(\vec{U}_{e}-\vec{d}^{h},\vec{v}_{e},\vec{\omega}_{e})(p^{h}-p)\, d\Omega^{h}=0.
	\end{multline*}
	The orthogonality implies 
	\begin{align*}
	\norm{(\vec{U}_{e},\vec{v}_{e},\vec{\omega}_{e})}_{E_{h}}-\norm{(\vec{d}^{h},0,0)}_{E_{h}} & \le\norm{(\vec{U}_{e}-\vec{d}^{h},\vec{v}_{e},\vec{\omega}_{e})}_{E_{h}}\\
	& \le\norm{(\vec{\Lambda}_{\vec{U}}-\vec{d}^{h},\vec{\Lambda}_{\vec{v}},\vec{\Lambda}_{\vec{\omega}})}_{E_{h}}\\
	& \le\norm{(\vec{\Lambda}_{\vec{U}},\vec{\Lambda}_{\vec{v}},\vec{\Lambda}_{\vec{\omega}})}_{E_{h}}+\norm{(\vec{d}^{h},0,0)}_{E_{h}}
	\end{align*}
	The estimation for $\vec{d}^{h}$ in \eqref{eq-dh} implies that for
	$d=2$ we have 
	\begin{align*}
	\norm{(\vec{d}^{h},0,0)}_{E_{h}}^{2} & =\sum_{i,j}\frac{1}{2}\left(\rho_{i+\frac{1}{2},j}\, H_{i+\frac{1}{2},j}\,(d_{i+\frac{1}{2},j}^{h})^{2}+\rho_{i,j+\frac{1}{2}}\, H_{i,j+\frac{1}{2}}(d_{i,j+\frac{1}{2}}^{h})^{2}\right)h^{2}\\
	& =\sum_{H_{i+\frac{1}{2},j}=H_{i,j+\frac{1}{2}}=1}O(h^{4})\, h^{2}+\sum_{0<H_{i+\frac{1}{2},j},\; H_{i,j+\frac{1}{2}}<1}O(h^{2})\, h^{2}\\
	& =O(h^{4})+O(h^{3})=O(h^{3})
	\end{align*}
	and for $d=3$ we similarly have 
	\begin{multline*}
	\norm{(\vec{d}^{h},0,0)}_{E_{h}}^{2}=\sum_{H_{i+\frac{1}{2},j,k}=H_{i,j+\frac{1}{2},k}=H_{i,j,k+\frac{1}{2}}=1}O(h^{4})\, h^{3}\\
	+\sum_{0<H_{i+\frac{1}{2},j,k},H_{i,j+\frac{1}{2},k},H_{i,j,k+\frac{1}{2}}<1}O(h^{2})\, h^{3}=O(h^{4})+O(h^{3})=O(h^{3}).
	\end{multline*}
	Here, we used the fact that the number of inside edges, where $H=1$,
	grows as $O(h^{d})$ and that of edges near the boundary, where $0<H<1,$
	grows as $O(h^{d-1})$.
	
	And Lemma \ref{lemma-n-GH-J-Jh} (iii) shows 
	\begin{align*}
	m\vec{\Lambda}_{\vec{v}} & =\int_{\Gamma}p\vec{n}\, dS+\int(p\dG H)\, d\Omega^{h}=\sum_{\vec{x}_{h}\in\Omega_{\Gamma}^{h}}\;\int_{\Gamma\cap C_{\vec{x}_{h}}}(p-p_{\vec{x}_{h}})\vec{n}\, dS=O(h)
	\end{align*}
	and 
	\[
	\bI\,\vec{\Lambda}_{\vec{\omega}}=\int_{\Gamma}p\vec{J}\, dS+\intop p\vec{J}^{h}\, d\Omega^{h}=O(h).
	\]

	On the other hand, the standard central finite difference operator
	$G$ gives 
	\[
	\dG p-\nabla p=\rho\,\vec{\Lambda}_{\vec{U}}=O(h^{2}).
	\]
	Combining all the estimations for $\vec{\Lambda}_{\vec{U}}$, $\vec{\Lambda}_{\vec{v}}$,
	and $\vec{\Lambda}_{\vec{\omega}}$, we conclude that 
	\begin{align*}
	\norm{(\vec{\Lambda}_{\vec{U}},\vec{\Lambda}_{\vec{v}},\vec{\Lambda}_{\vec{\omega}})}_{E_{h}}^{2} & =\frac{1}{2}\int H(\dG p-\nabla p)\cdot(\dG p-\nabla p)\, d\Omega^{h}\\
	& \qquad+\frac{1}{2}m\vec{\Lambda}_{v}\cdot\vec{\Lambda}_{v}+\frac{1}{2}\vec{\Lambda}_{\omega}\cdot\bI\vec{\Lambda}_{\omega}\\
	& =O(h^{4})+O(h^{2})+O(h^{2})=O(h^{2}).
	\end{align*}
	Consequently, we have 
	\begin{align*}
	\norm{(\vec{U}_{e},\vec{v}_{e},\vec{\omega}_{e})}_{E_{h}} & \le\norm{(\vec{\Lambda}_{\vec{U}},\vec{\Lambda}_{\vec{v}},\vec{\Lambda}_{\vec{\omega}})}_{E_{h}}+2\norm{(\vec{d}^{h},0,0)}_{E_{h}}\\
	& =O(h)+O(h^{1.5})=O(h),
	\end{align*}
	which proves the theorem. \end{proof}

\section{Numerical test}\label{sec6}

\subsection{Orthogonality and stability }
We showed in Theorem \ref{thm-discrete-energy} that the augmented
Hodge projection satisfies the orthogonality and the stability, and
numerically validate both of the properties in this example. In $[-2,2]\times[-4,4]$,
a solid with mass $m=4$ and inertia $\bI=2$ is located in a ball
centered at $(0,0)$ with radius $1$. Fluid with density $\rho=1$
fills the rectangle outside the solid. In this setting, the augmented
projection is performed on $\vec{U}^{*}=(0,-1)$, $\vec{v}^{*}=(\cos\theta,\sin\theta)$,
and $\vec{\omega}^{*}=0$ for each angle $\theta\in[0,2\pi]$. On
the boundary of the rectangle, the periodic boundary condition is
imposed to focus on the interface between fluid and solid.
Figure~\ref{fig:Streamlines_example1} depicts the velocity fields
before and after the projection, and the streamlines in the case $\theta=0$.
Note that the velocity fields before the projection does not satisfy
the non-penetration condition $\vec{U}\cdot\vec{n}=\vec{U}^{solid}\cdot\vec{n}$.
After the projection, the solid vector field is clearly uniform and
the non-penetration condition is now satisfied.

\begin{figure}
	\label{fig:Streamlines_example1} 
	
	\begin{centering}
		\includegraphics[scale=0.14]{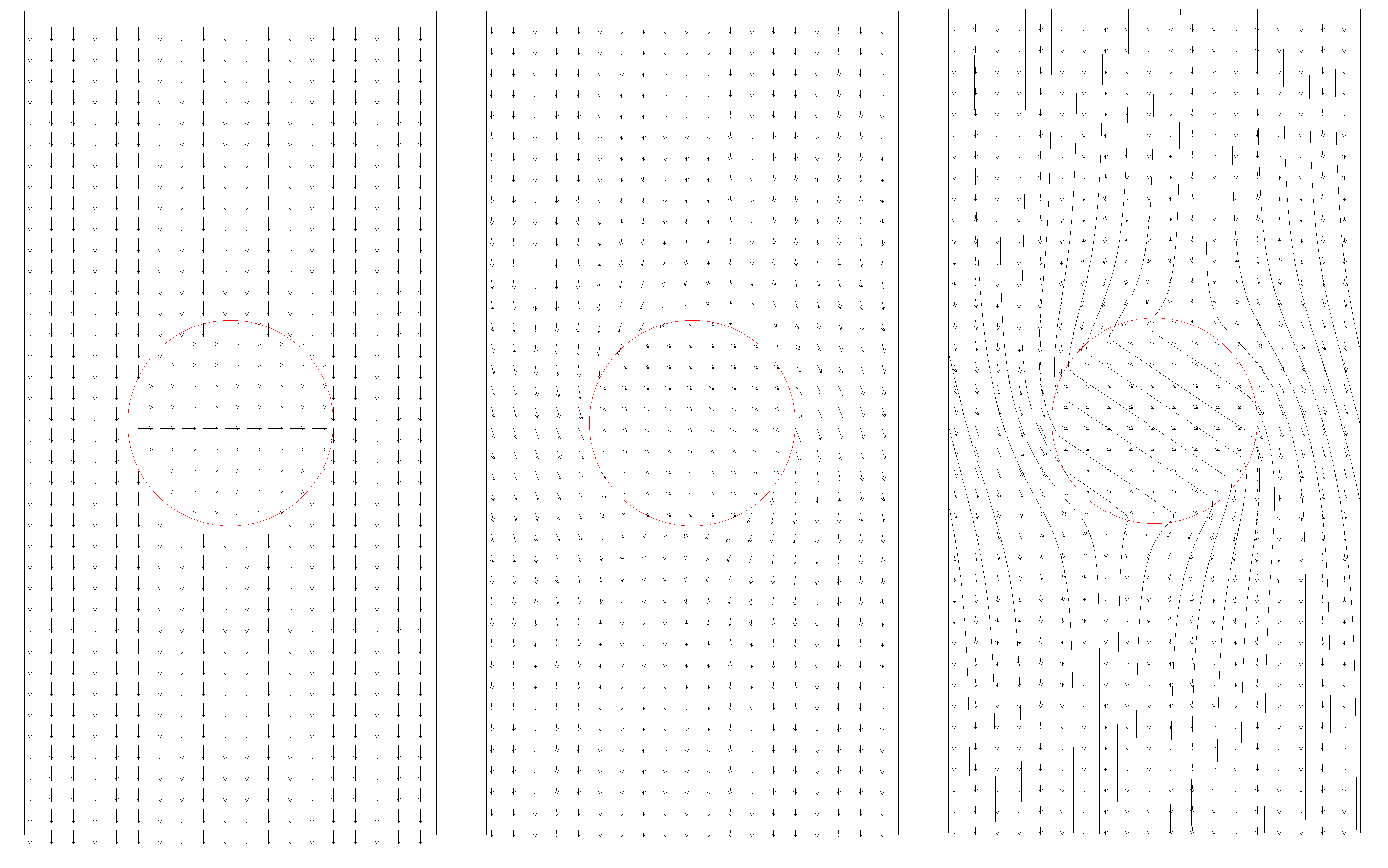} 
		\par\end{centering}
	
	\caption{The velocity fields of fluid and solid before (left) and after (middle)
		the augmented Hodge projection, and the streamlines (right) after
		projection.}
\end{figure}

Figure \ref{fig:Stability_example1} validates the theorem that the
orthogonality condition 
\[
\ip{(\vec{U}^{h},\vec{v}^{h},\vec{\omega}^{h}),(\vec{U}^{*}-\vec{U}^{h},\vec{v}^{*}-\vec{v}^{h},\vec{\omega}^{*}-\vec{\omega}^{h})}_{E_{h}}=0
\]
and the stability condition $\norm{(\vec{U}^{h},\vec{v}^{h},\vec{\omega}^{h})}_{E_{h}}\le\norm{(\vec{U}^{*},\vec{v}^{*},\vec{\omega}^{*})}_{E_{h}}$
are satisfied for all $\theta\in\left[0,2\pi\right]$. The numerical
solution $(\vec{U}^{h},\vec{v}^{h},\vec{\omega}^{h})$ was computed
in a uniform grid $40\times80$.

\begin{figure}
	\label{fig:Stability_example1} 
	
	\begin{centering}
		\includegraphics[scale=0.67]{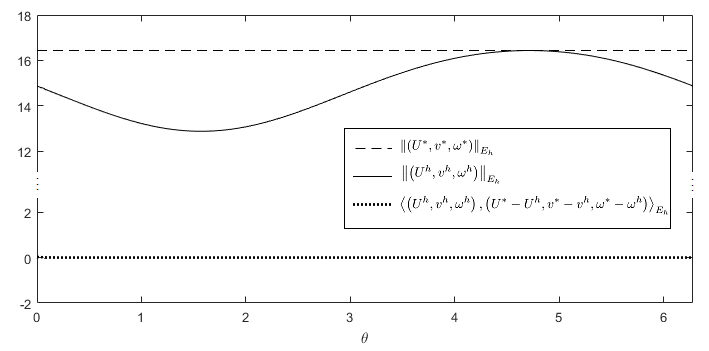} 
		\par\end{centering}
	
	\caption{The kinetic energy before (dashed line) and after (solid line) the
		augmented Hodge projection, and the inner product. This figure validates
		$\norm{(\vec{U}^{h},\vec{v}^{h},\vec{\omega}^{h})}_{E_{h}}\le\norm{(\vec{U}^{*},\vec{v}^{*},\vec{\omega}^{*})}_{E_{h}}$
		and $\ip{(\vec{U}^{h},\vec{v}^{h},\vec{\omega}^{h}),(\vec{U}^{*}-\vec{U}^{h},\vec{v}^{*}-\vec{v}^{h},\vec{\omega}^{*}-\vec{\omega}^{h})}_{E_{h}}=0$.}
\end{figure}

\subsection{Convergence in two dimensions}
We showed in Theorem \ref{them-convergence} that $$\norm{(\vec{U}^{h},\vec{v}^{h},\vec{\omega}^{h})-(\vec{U},\vec{v},\vec{\omega}}_{E_{h}}=O(h),$$
which implies that each of the approximations $\vec{U}^{h}\simeq\vec{U}$,
$\vec{v}^{h}\simeq\vec{v}$, and $\vec{\omega}^{h}\simeq\vec{\omega}$
is at least first-order accurate. In this example, we numerically
validate the convergence order. Inside $\left[-\frac{\pi}{2},\frac{\pi}{2}\right]^{2}$,
a solid with mass $m=1$ and inertia $\bI=\bE$ is located in a domain
$\left\{ \left(x,y\right)|\cos x\cos y<\frac{\sqrt{3}}{2}\right\} $.
Fluid with density $\rho=1$ fills the rectangle outside the solid.
In this setting, the augmented Hodge projection is performed on $\vec{U}^{*}=\vec{U}+\frac{1}{\rho}\nabla p$,
$\vec{v}^{*}=0-\frac{1}{m}\intop_{\Gamma}p\vec{n}ds$, and $\vec{\omega}^{*}=0-\bI^{-1}\intop_{\Gamma}p\vec{J}ds$
with $\vec{U}\left(x,y\right)=\left(\cos x\sin y,-\sin x\cos y\right)$
and $p\left(x,y\right)=e^{-\left(x-1\right)^{2}+y}$.

Note that $\vec{U}$ satisfies the incompressibility condition $\nabla\cdot\vec{U}=0$
and the non-penetration condition $\vec{U}\cdot\vec{n}=0$ on the
rectangular boundary as well as on the interface $\Gamma=\{(x,y) | \cos x\cos y=\frac{\sqrt{3}}{2}\}.$ Numerically computed solutions $(\vec{U}^{h},\vec{v}^{h},\vec{\omega}^{h})$ are compared to the exact solution $\left(\vec{U},\vec{v}=0,\vec{\omega}=0\right).$ The exact values of the boundary integrals are $$
\int_{\Gamma}p\vec{n}\, dS=[-0.61757657740494\cdots,-0.35222653922478\cdots]
$$ and $$\int_{\Gamma}p\vec{J}\, dS=-0.0000757711760965\cdots.$$
Table \ref{tab:Convergence-in-two}  reports the numerical results. The convergence order of $\norm{(\vec{U}^{h},\vec{v}^{h},\vec{\omega}^{h})-(\vec{U},\vec{v},\vec{\omega})}_{E_h}$ with respect to $h$ fluctuates, but its least squares fit, as plotted in Figure \ref{fig:Least-square-fit_graph}, indicates that the convergence order is $1.61.$ In the following estimate obtained from combining Lemma \ref{lemma-n-GH-J-Jh} (iii) and the last inequality in the proof of Theorem \ref{them-convergence}, 
\[
\norm{(\vec{U}-\vec{U}^{h},\vec{v}-\vec{v}^{h},\vec{\omega}-\vec{\omega}^{h})}_{E_{h}}
\le\frac{\abs{\Gamma}\left((\diam\Gamma)+1\right)\norm{\nabla p}_{L^{\infty}}}{2}\, h+O(h^{1.5}),
\]
the first term becomes dominant and forms an upper bound of the error for sufficiently small $h.$ All the errors in Figure \ref{fig:Least-square-fit_graph} are certainly below the upper bound, which validates the theorem.
We deduced from the theorem that each of $\vec U^h\simeq \vec U, \vec v^h \simeq v, \vec \omega^h\simeq \vec \omega$ is at least first order accurate. Table \ref{tab:individual error} confirms the deduction.

\begin{table}
	\begin{centering}
		\begin{tabular}{|c|c|c|c|}
			\hline 
			grid  &  h & $\norm{(\vec{U}-\vec{U}^{h},\vec{v}-\vec{v}^{h},\vec{\omega}-\vec{\omega}^{h})}_{E_{h}}$  & order\tabularnewline
			\hline 
			\hline 
			$20^{2}$  & $\frac{\pi}{20}$ & $3.14\times10^{-2}$  & \tabularnewline
			\hline 
			$40^{2}$  & $\frac{\pi}{40}$ & $5.01\times10^{-3}$  & 2.65\tabularnewline
			\hline 
			$80^{2}$  & $\frac{\pi}{80}$ & $3.42\times10^{-3}$  & 0.551\tabularnewline
			\hline 
			$160^{2}$  & $\frac{\pi}{160}$ & $5.70\times10^{-4}$  & 2.58\tabularnewline
			\hline 
			$320^{2}$  & $\frac{\pi}{320}$ &  $4.17\times10^{-4}$  & 0.451\tabularnewline
			\hline 
			$640^{2}$  & $\frac{\pi}{640}$ &  $6.99\times10^{-5}$  & 2.57\tabularnewline
			\hline 
		\end{tabular}
		\par\end{centering}
	
	\caption{Convergence order in the example of two dimensions. The order fluctuates, and its least squares fitting is tried in Figure \ref{fig:Least-square-fit_graph}.  \label{tab:Convergence-in-two}}
\end{table}

\begin{figure}

	\begin{centering}
		\includegraphics[scale=0.5]{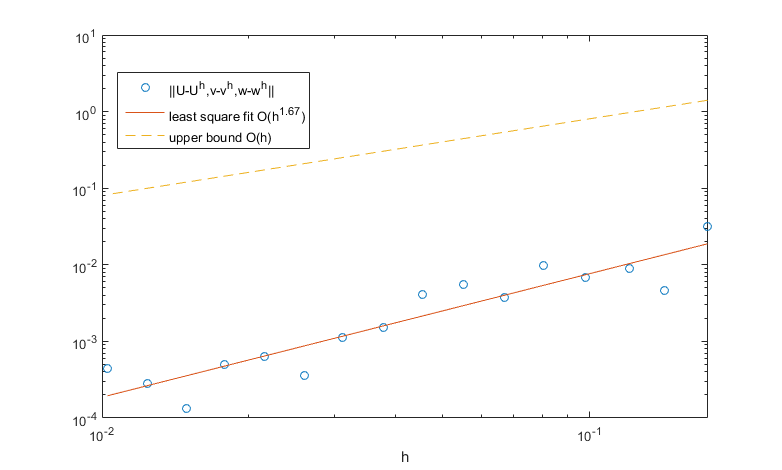} 
		\par\end{centering}
	
	\caption{ Log-plot of the least squares fit of the errors 
		$
		\|\left(\vec U^h, \vec v^h, \vec \omega^h\right) - \left(\vec U, \vec v, \vec \omega\right)\|_{E_h}
		$ in the example of two dimensions. The errors are certainly below the upper bound $O(h).$
		\label{fig:Least-square-fit_graph}}
\end{figure}

\begin{table}
	\begin{centering}
		\begin{tabular}{|c|c|}
			\hline 
			Individual error  & order\tabularnewline
			\hline 
			\hline 
			$\norm{\vec{U}-\vec{U}^{h}}_{L^{2}}$  & $O(h^{1.61})$\tabularnewline
			\hline 
			$\abs{\vec{v}-\vec{v}^{h}}$  & $O(h^{1.73})$\tabularnewline
			\hline 
			$\abs{\vec{\omega}-\vec{\omega}^{h}}$  & $O(h^{2.72})$\tabularnewline
			\hline 
		\end{tabular}
		\par\end{centering}
	
	\caption{The convergence order of each individual error in the sum $\|\left(\vec{U}^{h},\vec{v}^{h},\vec{\omega}^{h}\right)-\left(\vec{U},\vec{v},\vec{\omega}\right)\|_{E_h}^2=\frac{\rho}{2}\|\vec U^h-\vec U\|_{L^2}^2+\frac{m}{2}\|\vec v^h -\vec v\|^2+\frac{\bI}{2}\|\vec \omega^h -\vec \omega\|^2.$  For the same data in Figure \ref{fig:Least-square-fit_graph}, the order was obtained from the least squares fit of each error.\label{tab:individual error}}
\end{table}

\subsection{Convergence in three dimensions}
We measure the convergence order of the discrete Hodge projection
in a three-dimensional example.  In $[-2,2]^{2}\times[-4,4]$,
a solid with mass $m=\frac{8}{3}\pi$ and inertia $\bI=\frac{16}{15}\pi\,\mathbb{E}$
is located in a ball centered at $(0,0,0)$ of radius one. We recall that $\mathbb{E}$ is the $3\times 3$ identity matrix. Fluid with density $\rho=1$ fills the box outside the solid. In this setting, the augmented Hodge 
projection is performed on $\vec{U}^{*}=(0,0,-1)$, $\vec{v}^{*}=(0,0,-1)$, and $\vec{\omega}^{*}=(0,0,0)$.
The non-penetration condition is imposed on the rectangular boundary. Figure \ref{fig:three_dimension_setting} illustrates the setting. 

\begin{figure}
	\begin{centering}
		\includegraphics[scale=0.33]{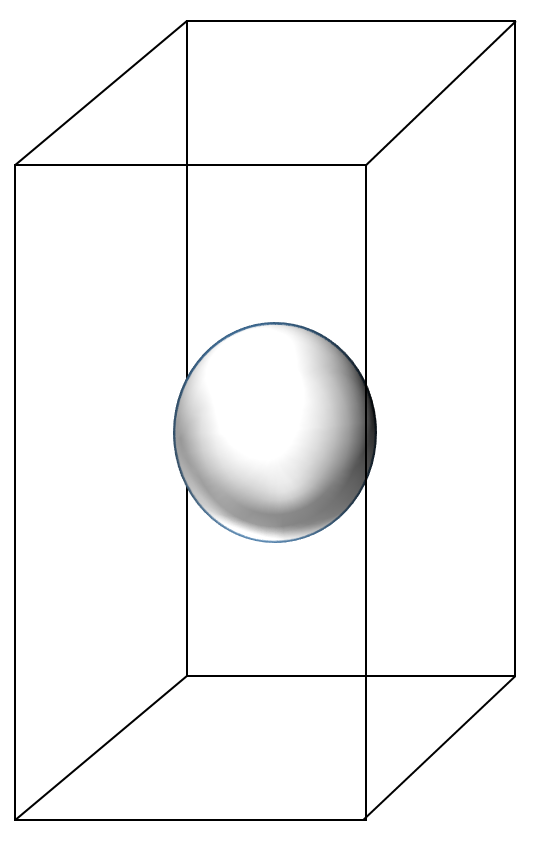}
		\includegraphics[scale=0.35]{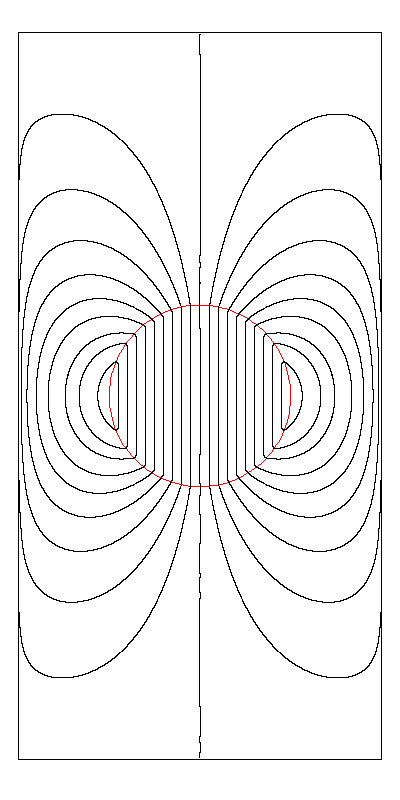} 
		\par\end{centering}
	
	\caption{The physical setting for the example of three dimensions (left) and the streamlines after the projection in section ${x=0}$ (right). \label{fig:three_dimension_setting}}
\end{figure}

Since the exact solution is unknown, we instead measure the convergence order through $\norm{(\vec{U}^{2h}-\vec{U}^{h},\vec{v}^{2h}-\vec{v}^{h},\vec{\omega}^{2h}-\vec{\omega}^{h})}_{E_{h}}$,
based on the fact that we have
\begin{align*}
\norm{(\vec{U}^{2h}-\vec{U}^{h},\vec{v}^{2h}-\vec{v}^{h},\vec{\omega}^{2h}-\vec{\omega}^{h})}_{E_{h}} & \le\norm{(\vec{U}-\vec{U}^{2h},\vec{v}-\vec{v}^{2h},\vec{\omega}-\vec{\omega}^{2h})}_{E_{h}}\\
& \qquad+\norm{(\vec{U}-\vec{U}^{h},\vec{v}-\vec{v}^{h},\vec{\omega}-\vec{\omega}^{h})}_{E_{h}}\\
& =O(h^{\alpha}+2^{\alpha}h^{\alpha})=O(h^{\alpha}),
\end{align*}
whenever $\norm{(\vec{U}-\vec{U}^{h},\vec{v}-\vec{v}^{h},\vec{\omega}-\vec{\omega}^{h})}_{E_{h}}=O(h^{\alpha})$ for the exact solution $\left(\vec U,\vec v,\vec \omega \right)$.

\begin{table}
	\begin{centering}
		\begin{tabular}{|c|c|c|c|}
			\hline 
			grid  & $h$  & $\norm{(\vec{U}^{2h},\vec{v}^{2h},\vec{\omega}^{2h})-(\vec{U}^{h},\vec{v}^{h},\vec{\omega}^{h})}_{E_{h}}$  & order\tabularnewline
			\hline 
			\hline 
			$16^{2}\times32$  & $\frac{1}{4}$  & $6.16\times10^{-2}$  & \tabularnewline
			\hline 
			$32^{2}\times64$  & $\frac{1}{8}$  & $1.62\times10^{-2}$  & 1.92\tabularnewline
			\hline 
			$64^{2}\times128$  & $\frac{1}{16}$  & $4.31\times10^{-3}$  & 1.91\tabularnewline
			\hline 
			$128^{2}\times256$  & $\frac{1}{32}$  & $1.26\times10^{-3}$  & 1.77\tabularnewline
			\hline 
			$256^{2}\times512$  & $\frac{1}{64}$  & $3.74\times10^{-4}$  & 1.75\tabularnewline
			\hline 
		\end{tabular}
	\end{centering}
	
	\caption{The convergence error 
		in the example of three dimensions\label{tab:three_dimension_sum}}
\end{table}

Table \ref{tab:three_dimension_sum} reports the numerical results, which validates the convergence order $O\left(h\right)$ shown in Theorem \ref{them-convergence}. Also the table shows that the order is actually far greater than one. Table \ref{tab:three_dimension_indi} indicates that $\norm{\vec{U}^{2h}-\vec{U}^{h}}_{L^{2}}=O(h^{1.6})$ and $\abs{\vec{v}^{h}-\vec{v}^{2h}}+\abs{\vec{\omega}^{h}-\vec{\omega}^{2h}}=O(h^2)$. The convergence error $\norm{(\vec{U}^{2h},\vec{v}^{2h},\vec{\omega}^{2h})-(\vec{U}^{h},\vec{v}^{h},\vec{\omega}^{h})}_{E_{h}}$ can be regarded as a sum of $\norm{\vec{U}^{2h}-\vec{U}^{h}}_{L^{2}}$ and $\abs{\vec{v}^{h}-\vec{v}^{2h}}+\abs{\vec{\omega}^{h}-\vec{\omega}^{2h}}$, when the order calculation is a matter of interests. It implies that $\norm{(\vec{U}^{2h},\vec{v}^{2h},\vec{\omega}^{2h})-(\vec{U}^{h},\vec{v}^{h},\vec{\omega}^{h})}_{E_{h}}$ is mainly influenced by $\norm{\vec{U}^{2h}-\vec{U}^{h}}_{L^{2}}=O(h^{1.6})$ rather than by the other term, when $h$ becomes smaller. This explains why the convergence order in Table \ref{tab:three_dimension_sum} slowly decreases toward $1.6$.

\begin{table}
	\begin{centering}
		\begin{tabular}{|c|c|c|c|c|c|}
			\hline 
			grid  & $h$  & $\norm{\vec{U}^{2h}-\vec{U}^{h}}_{L^{2}}$  & order  & $\abs{\vec{v}^{h}-\vec{v}^{2h}}+\abs{\vec{\omega}^{h}-\vec{\omega}^{2h}}$  & order\tabularnewline
			\hline 
			\hline 
			$16^{2}\times32$  & $\frac{1}{4}$  & $2.65\times10^{-2}$  &  & $5.56\times10^{-2}$  & \tabularnewline
			\hline 
			$32^{2}\times64$  & $\frac{1}{8}$  & $8.74\times10^{-3}$  & 1.60  & $1.36\times10^{-2}$  & 2.03\tabularnewline
			\hline 
			$64^{2}\times128$  & $\frac{1}{16}$  & $2.61\times10^{-3}$  & 1.74  & $3.43\times10^{-3}$  & 1.98\tabularnewline
			\hline 
			$128^{2}\times256$  & $\frac{1}{32}$  & $9.30\times10^{-4}$  & 1.48  & $8.62\times10^{-4}$  & 1.99\tabularnewline
			\hline 
			$256^{2}\times512$  & $\frac{1}{64}$  & $3.50\times10^{-4}$  & 1.60  & $2.16\times10^{-4}$  & 2.00\tabularnewline
			\hline 
		\end{tabular}
	\end{centering}
	\caption{The convergence errors of $\norm{\vec{U}^{2h}-\vec{U}^{h}}_{L^{2}}$  and $\abs{\vec{v}^{h}-\vec{v}^{2h}}+\abs{\vec{\omega}^{h}-\vec{\omega}^{2h}}$ in the example of three dimensions \label{tab:three_dimension_indi}}
\end{table}

\section{Conclusion and comment}\label{sec7}
In this work, we have studied a Fluid-Solid interaction by taking a monolithic treatment on fluid-solid interaction. It is based on the fact  that fluid velocity field $\vec U$ and solid velocities $\vec v$ and $\vec \omega$ do not separately proceed but as a whole by a combined state variable $(\vec U, \vec v, \vec \omega)$.
We introduced the so-called augmented Hodge decomposition of the state variable into two orthogonal components, which is a variation of the Hodge decomposition.
The decomposition enables us to decouple the computations of the velocity
and the pressure in the incompressible Navier-Stokes equation.
Then, the decomposition is fulfilled by solving an elliptic equation for the pressure with non-local Robin type boundary condition.
We have shown the existence, uniqueness and the regularity of the solution
to the equation. The monolithic treatment leads to the stability that
the kinetic energy does not increase in the projection step. Using an Heaviside function, we expressed the boundary condition independent of the interface. We proposed a numerical method producing the numerical solution at least with first order accuracy. Also, we showed that the unique decomposition and orthogonality also hold in the discrete setting. We carried out numerical experiments for 2 and 3 dimensions and the numerical tests validate our analysis and arguments. Even though the experiments supports our arguments, $\norm{(\vec{U}^{h},\vec{v}^{h},\vec{\omega}^{h})-(\vec{U},\vec{v},\vec{\omega})}_{E_h}=O(h),$ it reveals that the convergence order is greater than one and the error is mainly influenced by $\norm{\vec{U}^{2h}-\vec{U}^{h}}_{L^{2}}=O(h^{1.6})$ rather than by the other term, when $h$ becomes smaller. The analysis for the phenomenon would involve very different arguments from the current one, and we put it off to future work.

\end{document}